\documentclass[10pt]{amsart}

\usepackage[english]{babel}
\usepackage[T1]{fontenc}
\usepackage[utf8]{inputenc}
\usepackage{amsmath}
\usepackage{amssymb}
\usepackage{amsthm}
\usepackage{color}
\usepackage[a4paper]{geometry}
\usepackage{tikz}
\usepackage{mathdots}
\usepackage[all]{xy}
\usepackage{hyperref}
\usepackage{accents}
\usepackage{lmodern}
\usepackage[multiple]{footmisc}
\usepackage{mathrsfs}

\theoremstyle{definition}
\newtheorem{definition}{Definition}[section]

\theoremstyle{plain}
\newtheorem{theorem}[definition]{Theorem}
\newtheorem{prop}[definition]{Proposition}
\newtheorem{cor}[definition]{Corollary}
\newtheorem{lem}[definition]{Lemma}

\newcommand{\mb}{\mathbb}
\newcommand{\mc}{\mathcal}
\newcommand{\mf}{\mathfrak}

\newcommand{\bs}{\boldsymbol}
\newcommand{\Nm}{\textup{Nm}}
\newcommand{\Vol}{\textup{Vol}}

\newcommand{\diam}{\textup{diam}}

\begin{document}

\title{On a counting theorem for weakly admissible lattices}

\author{Reynold Fregoli}
\address{Department of Mathematics\\ 
Royal Holloway, University of London\\ 
TW20 0EX Egham\\ 
UK}
\email{Reynold.Fregoli.2017@live.rhul.ac.uk}

\subjclass{Primary ; Secondary }
\date{\today, and in revised form ....}

\dedicatory{}


\begin{abstract}
We give a precise estimate for the number of lattice points in certain bounded subsets of $\mb{R}^{n}$ that involve ``hyperbolic spikes''
and occur naturally in multiplicative Diophantine approximation.
We use Wilkie's o-minimal structure $\mb{R}_{\exp}$ and expansions thereof to 
formulate our counting result in a general setting. 
We give two different applications of our counting result.
The first one establishes nearly sharp upper bounds for sums of reciprocals of fractional parts, and
thereby sheds light on a question raised by L\^e and Vaaler, extending previous work of Widmer and of the author. 
The second application establishes new examples of linear subspaces of Khintchine type thereby refining  a theorem by Huang and Liu.
For the proof of our counting result we develop a sophisticated partition method which is crucial for further upcoming work
on sums of reciprocals of fractional parts over distorted boxes.
\end{abstract}

\maketitle

\section{Introduction}

\subsection{Notation}

Let $X$ be a set. For any pair of functions $f,g:X\to\mb{R}$, we write $f\ll g$ ($f\gg g$) to mean that there exists a real number $c>0$ such that $f(x)\leq cg(x)$ ($f(x)\geq cg(x)$) for all $x\in X$. If the constant $c$ depends on any parameters, we write them under the symbol $\ll$ ($\gg$). We write $O_{c}(f)$ to indicate a function $g$ such that $g\ll_{c}f$. We use $|\cdot|_{2}$ to denote the Euclidean norm on $\mb{R}^{n}$ and $|\cdot|_{\infty}$ to denote the maximum norm. We write $\mb{N}$ for the set $\{1,2,3,\dotsc\}$ of positive integers. We indicate by $\|x\|$ the distance from any $x\in\mb{R}$ to the nearest integer, i.e., $\min\{|x-n|:n\in\mb{Z}\}$. We denote by $\diam\,X$ the diameter (i.e., $\sup\{|x-y|:x,y\in X\}$) of any set $X\subset\mb{R}^{n}$, and we use $\Vol_{d}(X)$ to indicate its $d$-dimensional Hausdorff measure ($d\in\mb{N}$). When the dimension $d$ is not specified, we assume $d=n$. 

\subsection{Main result}
\label{subsec:weakadm}

In this paper we prove a general counting result for weakly admissible lattices. More specifically, we estimate the number of lattice points lying in the area bounded by a certain compact hypersurface defined in terms of the lattice structure. We generalise this result to any definable set in Wilkie's o-minimal structure $\mb{R}_{\exp}$ lying within the above mentioned hypersurface, and we derive an asymptotic formula for the number of lattice points contained in any such set. Our counting principle allows us to shed light on a question raised by L\^e and Vaaler on the behaviour of certain sums of reciprocals of fractional parts. It also yields a refinement of a theorem proved by Huang and Liu on linear subspaces of Khintchine type.

Before stating the main result, we look at a special case that already captures the main features of our counting principle.  
Let $\mc{M},\mc{N}\in\mb{N}$ and let $\bs{L}\in\mb{R}^{\mc{M}\times \mc{N}}$. We denote by $L_{1},\dotsc,L_{\mc{M}}:\mb{R}^{\mc{N}}\to\mb{R}$ the linear forms induced by the rows of the matrix $\bs{L}$, i.e., the functions
$$L_{i}(x_{1},\dotsc,x_{\mc{N}}):=\sum_{j=1}^{\mc{N}}L_{ij}x_{j}$$
for $i=1,\dotsc, \mc{M}$. We assume throughout the paper that $1$ along with the coefficients $L_{i1},\dotsc,L_{i\mc{N}}$ of each of these linear forms are linearly independent over $\mb{Q}$.
Let $\varepsilon,T\in(0,+\infty)$ and let $Q\in[1,+\infty)$. We consider the set
\begin{multline*}
M(\bs{L},\varepsilon,T,Q):=\left\{(\bs{p},\bs{q})\in\mb{Z}^{\mc{M}}\times\left(\mb{Z}^{\mc{N}}\setminus\{\bs{0}\}\right):\prod_{i=1}^{\mc{M}}\left|L_{i}(\bs{q})+p_{i}\right|<\varepsilon,\right.\\
|L_{i}(\bs{q})+p_{i}|\leq T\ i=1,\dotsc,\mc{M},\ |q_{j}|\leq Q\ j=1,\dotsc,\mc{N}\Bigg\}.
\end{multline*} 
Our goal is to estimate the cardinality of $M(\bs{L},\varepsilon,T,Q)$. To this end, we let
$$ \bs{A}_{\bs{L}}:=\left(\begin{array}{@{}c|c@{}}
    \bs{I}_{\mc{M}} & \!\!\!\!\bs{L} \\\hline
    \bs{0} & \accentset{\phantom{a}}{\bs{I}}_{\mc{N}}
  \end{array}\right)\in\mb{R}^{(\mc{M}+\mc{N})\times(\mc{M}+\mc{N})},$$
where $\bs{I}_{\mc{M}}$ and $\bs{I}_{\mc{N}}$ are identity matrices of size $\mc{M}$ and $\mc{N}$ respectively, and we let $\Lambda_{\bs{L}}:=\bs{A}_{\bs{L}}\mb{Z}^{\mc{M}+\mc{N}}$. We also set
$$Z:=\left\{\bs{x}\in\mb{R}^{\mc{M}}:\prod_{i=1}^{\mc{M}}\left|x_{i}\right|<\varepsilon,\ |x_{i}|\leq T\ i=1,\dotsc,\mc{M}\right\}\times[-Q,Q]^{\mc{N}}.$$
Then, we have
$$\# M(\bs{L},\varepsilon,T,Q)=\#\left((\Lambda_{\bs{L}}\cap Z)\setminus C\right),$$
where $C:=\left\{(\bs{x},\bs{y})\in\mb{R}^{\mc{M}+\mc{N}}:\bs{y}=\bs{0}\right\}$. Therefore, estimating $\# M(\bs{L},\varepsilon,T,Q)$ is equivalent to estimating $\#(\Lambda_{\bs{L}}\cap Z)$, if we exclude the points of $\Lambda_{\bs{L}}$ that lie in $C$.

We now make some assumptions on the lattice $\Lambda_{\bs{L}}$. First, we assume that the distance between the points in $\Lambda_{\bs{L}}\setminus C$ and the coordinate subspaces of $\mb{R}^{\mc{M}+\mc{N}}$ orthogonal to $C$ is always positive. In the worst case, this distance will be decaying as we move away from the origin. We want to control its decay rate in terms of the distance from the origin. Hence, we additionally assume that the distance between the points of $\Lambda_{\bs{L}}\setminus C$ and the coordinate subspaces orthogonal to $C$ is also bounded from below by a certain non-increasing function. To make this precise, we give the following definition.

\begin{definition}
\label{def:semimultbadapp0}
Let $\phi:[1,+\infty)\to(0,1]$ be a non-increasing\footnote{We say that $\phi$ is non-increasing if $\phi(x)\geq\phi(y)$ for all $x<y$.} function. We say that a matrix $\bs{L}\in\mb{R}^{\mc{M}\times\mc{N}}$ is $\phi$-semi multiplicatively badly approximable if
\begin{equation}
\label{eq:semimultbadapp0}
|\bs{q}|_{\infty}^{\mc{N}}\prod_{i=1}^{\mc{M}}\|L_{i}(\bs{q})\|\geq\phi(|\bs{q}|_{\infty})\nonumber
\end{equation}
for all $\bs{q}\in\mb{Z}^{\mc{N}}\setminus\{\bs{0}\}$. If the function $\phi$ can be chosen constant, we say that $\bs{L}$ is semi-multiplicatively badly approximable.
\end{definition}

For the lattice $\Lambda_{\bs{L}}$, the purely arithmetic property introduced in Definition \ref{def:semimultbadapp0} yields the geometric condition described above. Provided this geometric condition is fulfilled, we can derive an asymptotic estimate for $\#\left((\Lambda_{\bs{L}}\cap Z)\setminus C\right)$.

\begin{prop}
\label{cor:cor1}
Let $\bs{L}\in\mb{R}^{\mc{M}\times \mc{N}}$ be a $\phi$-semi multiplicative badly approximable matrix and suppose that $T^{\mc{M}}/\varepsilon\geq e^{\mc{M}}$, where $e=2.71828\dots$ is the base of the natural logarithm. Then, we have
\begin{equation}
\left|\#\left((\Lambda_{\bs{L}}\cap Z)\setminus C\right)-\Vol\,Z\right|\ll_{\mc{M},\mc{N}}(1+T)^{\mc{M}+\mc{N}-1}\log\left(\frac{T^{\mc{M}}}{\varepsilon}\right)^{\mc{M}-1}\left(\frac{\varepsilon Q^{\mc{N}}}{\phi(Q)}\right)^{\frac{\mc{M}+\mc{N}-1}{\mc{M}+\mc{N}}},\nonumber
\end{equation}
where
$$\Vol\,Z=2^{\mc{M}+\mc{N}} Q^{\mc{N}}\left[\varepsilon\log\left(\frac{T^{\mc{M}}}{\varepsilon}\right)^{\mc{M}-1}+T^{\mc{M}}\left(1-\left(1-\frac{\varepsilon}{T^{\mc{M}}}\right)^{\mc{M}-1}\right)\right].$$
\end{prop}

The case $\mc{M}=2$, $\mc{N}=1$ of Proposition \ref{cor:cor1} was proved by Widmer \cite{Widmer:AsymptoticDiophantine}. We briefly explain how the proof is structured, so that we can highlight the main difficulties. The key idea is to decompose the set $Z$ into approximately $\log\left(T/\varepsilon\right)^{\mc{M}-1}$ subsets. To each such subset we apply a different diagonal linear map and obtain a ball-like shaped set. We then count the points of the corresponding transformation of $\Lambda_{\bs{L}}$ lying in each of these sets. Note that for $\mc{M}=10$, $T=1$, and $\varepsilon=1/10$, we already find more than $1800$ different subsets and linear maps.
The finer the subdivision, the more precise is each single estimate. However, when the subdivision is too fine, we end up summing too many error terms, and controlling the minima of the transformed lattices becomes rather difficult.    
The geometric condition introduced in Definition \ref{def:semimultbadapp0} allows us to give sufficiently good bounds on the first successive minima of these lattices to control the total error term. The essence of this partition method is summarised in Proposition \ref{prop:partition}, which itself is a crucial ingredient in our upcoming work on sums of reciprocals of fractional parts over general boxes.

The strategy that we described above also applies to prove a much more general counting principle, that is the central object of this paper. In lieu of the lattice $\Lambda_{\bs{L}}$ and the set $Z$, we can consider a general weakly admissible lattice in $\mb{R}^{\mc{M}+\mc{N}}$ and a general definable set  contained in $Z$. In this much more general setting, we can prove an analogous counting result.

Before stating this result, we introduce some notation and we recall the main definitions. Let $L\in\mb{N}$ and let $\bs{l}\in\mb{N}^{L}$. We set $V_{\bs{l}}:=\prod_{h=1}^{L}\mb{R}^{l_{h}}$ and we write $\underline{\bs{v}}=(\bs{v}_{1},\dotsc,\bs{v}_{L})$ for any vector $\underline{\bs{v}}$ in $V_{\bs{l}}$. Note that each $\bs{\alpha}\in(0,+\infty)^{L}$ induces a multiplicative norm on the space $V_{\bs{l}}$, given by $\prod_{i=1}^{L}|\bs{v}_{i}|_{2}^{\alpha_{i}}$. We indicate this norm by $\Nm_{\bs{\alpha}}(\underline{\bs{v}})$. The following definition is a generalisation of the property considered in Definition \ref{def:semimultbadapp0}, due to Widmer \cite{Widmer:WeakAdmiss}.

\begin{definition}
\label{def:weakadmin}
Let
$$C:=\left\{\underline{\bs{v}}\in V_{\bs{l}}:\bs{v}_{i}=\bs{0}\ i\in I\right\},$$
for some $\emptyset\neq I\subset\{1,\dotsc,L\}$, and let $\mc{A}:=\alpha_{1}+\dotsb+\alpha_{L}$. We say that a full rank lattice $\Lambda\subset V_{\bs{l}}$ is weakly admissible for the couple $((\bs{l},\bs{\alpha}),C)$ if
$$\nu(\Lambda,\varrho):=\inf\left\{\textup{Nm}_{\bs{\alpha}}(\underline{\bs{v}})^{1/\mc{A}}:\underline{\bs{v}}\in\Lambda\setminus C,\ \left|\underline{\bs{v}}\right|_{2}<\varrho\right\}>0$$
for all $\varrho>0$, where we interpret $\inf\emptyset$ as $+\infty>0$.
\end{definition}

For our purpose, it is convenient to work with the product of two spaces of the form $V_{\bs{l}}$. We therefore adopt a double index notation.
Let $M,N\in\mb{N}$ and let $\mc{S}:=((\bs{m},\bs{n}),(\bs{\beta},\bs{\gamma}))$, where $\bs{m}\in\mb{N}^{M}$, $\bs{n}\in\mb{N}^{N}$, $\bs{\beta}\in(0,+\infty)^{M}$, and $\bs{\gamma}\in(0,+\infty)^{N}$. Let
$\mc{M}:=\sum_{i=1}^{M}m_{i}$ and let $\mc{N}:=\sum_{j=1}^{N}n_{j}$. Let also $\mc{B}:=\sum_{i=1}^{M}\beta_{i}$ and let $\mc{C}:=\sum_{j=1}^{N}\gamma_{j}$.
We consider the vector space $V:=V_{\bs{m}}\times V_{\bs{n}}:=\prod_{i=1}^{M}\mb{R}^{m_{i}}\times\prod_{j=1}^{N}\mb{R}^{n_{i}}$ and we denote its vectors by $(\underline{\bs{x}},\underline{\bs{y}})=(\bs{x}_{1},\dotsc,\bs{x}_{M},\bs{y}_{1},\dotsc,\bs{y}_{N})$. As mentioned above, the vectors $\bs{\beta}$ and $\bs{\gamma}$ induce a multiplicative norm on $V$, given by $\textup{Nm}_{(\bs{\beta},\bs{\gamma})}(\underline{\bs{x}},\underline{\bs{y}}):=\textup{Nm}_{\bs{\beta}}(\underline{\bs{x}})\textup{Nm}_{\bs{\gamma}}(\underline{\bs{y}})$, where $\textup{Nm}_{\bs{\beta}}(\underline{\bs{x}}):=\prod_{i=1}^{M}|\bs{x}_{i}|_{2}^{\beta_{i}}$ and $\textup{Nm}_{\bs{\gamma}}(\underline{\bs{y}}):=\prod_{j=1}^{N}|\bs{y}_{j}|_{2}^{\gamma_{j}}$. Throughout this section, we fix a subspace $C\subset V$ of the form
$$C:=\left\{(\underline{\bs{x}},\underline{\bs{y}})\in V:\bs{x}_{i}=\bs{0}\ i\in I,\bs{y}_{j}=\bs{0}\ j\in J\right\},$$
where $I\subseteq\{1,\dotsc,N\}$, $J\subseteq\{1,\dotsc,M\}$, and $I\cup J\neq\emptyset$.

Now, we introduce a generalisation of the set $Z$ appearing in Proposition \ref{cor:cor1}. We let
$$\mc{H}:=\left\{(\underline{\bs{x}},\varepsilon,T)\in V_{\bs{m}}\times(0,+\infty)^{2}:\textup{Nm}_{\bs{\beta}}(\underline{\bs{x}})^{\frac{1}{\mc{B}}}<\varepsilon,\ |\bs{x}_{i}|_{2}\leq T\ i=1,\dotsc,M\right\},$$
and
$$\mc{R}:=\left\{(\underline{\bs{y}},\bs{Q})\in V_{\bs{n}}\times\mb{R}^{N}:|\bs{y}_{j}|_{2}\leq Q_{j}\ j=1,\dotsc,N\right\}.$$
Then, we set $\mc{Z}:=\mc{H}\times\mc{R}$. Finally, for $\bs{Q}\in(0,+\infty)^{N}$ we define
$$Q:=\left(\prod_{j=1}^{N}Q_{j}^{\gamma_{j}}\right)^{1/\mc{C}},$$
and for any $\Gamma\subset V_{\bs{l}}$ we define $\lambda_{1}(\Gamma):=\inf\{|\bs{v}|_{2}:\bs{v}\in\Gamma\setminus\{\bs{0}\}\}$.

To make our counting result applicable to a large class of sets we use o-minimal structures. For the convenience of the reader we 
quickly recall the basic definitions. 

\begin{definition}
\label{def:definablestruc}
A structure over $\mb{R}$ is a sequence $\mathfrak{S}=(\mathfrak{S}_{n})_{n\in\mb{N}}$ of families of subsets of $\mb{R}^{n}$ such that for each $n$:\vspace{2mm}
\begin{itemize}
\item[$i)$] $\mathfrak{S}_{n}$ is a boolean algebra of subsets of $\mb{R}^{n}$ (under the usual set-theoretic operations);\vspace{2mm}
\item[$ii)$] $\mathfrak{S}_{n}$ contains every semi-algebraic subset of $\mb{R}^{n}$;\vspace{2mm}
\item[$iii)$] if $A\in\mathfrak{S}_{n}$ and $B\in\mathfrak{S}_{m}$, then $A\times B\in\mathfrak{S}_{n+m}$;\vspace{2mm}
\item[$iv)$] if $\pi:\mb{R}^{n+m}\to\mb{R}^{n}$ is the projection onto the first $n$ coordinates and $A\in\mathfrak{S}_{n+m}$, then $\pi(A)\in\mathfrak{S}_{n}$.\vspace{2mm}
\end{itemize}
A structure $\mathfrak{S}$ over $\mb{R}$ is said to be $o$-minimal if additionally:
\begin{itemize}
\item[$v)$] the boundary of any set in $\mathfrak{S}_{1}$ is finite.
\end{itemize}
\end{definition}

A set $S\subset\mb{R}^{n}$ is definable in the structure $\mathfrak{S}$ if $S\in\mathfrak{S}_{n}$. A map $f:A\to B$ is definable if its graph $\Gamma(f):=\{(\bs{x},f(\bs{x})):\bs{x}\in A\}\subset A\times B$ is a definable set.

Let $t\in\mb{N}$ and let $\mc{W}\subset V_{\bs{l}}\times\mb{R}^{t}$ be a definable set. 
We call $\mc{W}$ a definable family in $V_{\bs{l}}$, and we call the variables $\bs{\tau}\in\mb{R}^{t}$ parameters of $\mc{W}$. For $\bs{\tau}\in\mb{R}^{t}$ we call the set
$$W_{\bs{\tau}}:=\left\{\underline{\bs{v}}\in V_{\bs{l}}:(\underline{\bs{v}},\bs{\tau})\in\mc{W}\right\}$$
the fibre of $\mc{W}$ above $\bs{\tau}$.
In our setting, functions such as $f(x)=x^r=\exp(r\log x)$ (for any real r>0) and $\log x$ on $(0,+\infty)$ need to be definable.
Therefore, we require that the o-minimal structure $\mathfrak{S}$ we are working with extends Wilkie's o-minimal structure $\mb{R}_{\exp}$ \cite{Wilkie:Model}, i.e., we require that each set definable 
in $\mb{R}_{\exp}$ is also definable in $\mathfrak{S}$.

We recall that every subset of $V_{\bs{l}}\times\mb{R}^{t}$ of the form 
\begin{equation}
\mc{W}=\left\{(\underline{\bs{v}},\bs{\tau})\in V_{\bs{l}}\times\mb{R}^{t}:\mf{F}(\underline{\bs{v}},\bs{\tau})\geq\bs{0}\ (>\bs{0})\right\},\nonumber
\end{equation}
where $\mf{F}$ is a finite system of functions in the variables $\underline{\bs{v}}$ and $\bs{\tau}$, obtained by the (suitably interpreted) composition of polynomials, exponential functions $\exp:\mb{R}\to\mb{R}$, and logarithms $\log:(0,+\infty)\to\mb{R}$, is definable in $\mb{R}_{\exp}$.

From now on, we see the set $\mc{Z}$ as a definable family in $\mb{R}_{\exp}$, with parameters $\bs{\eta}:=(\varepsilon,T,\bs{Q})\in(0,+\infty)^{2+N}$. In analogy with the above, we indicate its fibres by $Z_{\bs{\eta}}$. We can now state our main theorem.

\begin{theorem}
\label{thm:mainestimate}
Let $\Lambda\subset V$ be a weakly admissible lattice for the couple $(\mc{S},C)$ and let $\mc{W}\subset V\times\mb{R}^{t}$ be definable in an o-minimal structure expanding 
$\mb{R}_{\exp}$. Suppose that for all $\bs{\tau}\in\mb{R}^{t}$ there exists $\bs{\eta}(\bs{\tau})=(\varepsilon,T,\bs{Q})\in(0,+\infty)^{2+N}$ such that $W_{\bs{\tau}}\subset Z_{\bs{\eta}(\bs{\tau})}$. Then, for all $\bs{\tau}\in\mb{R}^{t}$ and all choices of $\bs{\eta}(\bs{\tau})$ with $T/\varepsilon> e$ (where $e=2.71828\dots$ is the base of the natural logarithm) we have
\begin{multline}
\left|\#(\Lambda\cap W_{\bs{\tau}})-\frac{\Vol\,W_{\bs{\tau}}}{\det\Lambda}\right|\ll_{\mc{W},\bs{\beta},\bs{\gamma}} \\
\inf_{0<r\leq\diam(Z_{\bs{\eta}(\bs{\tau})})}\log\left(\frac{T}{\varepsilon}\right)^{M-1}\left(\frac{\left(\varepsilon^{\mc{B}} Q^{\mc{C}}\right)^{\frac{1}{\mc{B}+\mc{C}}}}{\nu(\Lambda,r)}+\frac{\diam\left(Z_{\bs{\eta}(\bs{\tau})}\right)}{r}+\frac{\diam\left(Z_{\bs{\eta}(\bs{\tau})}\cap C\right)}{\lambda_{1}(\Lambda\cap C)}\right)^{\mc{M}+\mc{N}-1}\nonumber.
\end{multline} 
\end{theorem}

We recall that the special case $\mc{M}=1$ of Theorem \ref{thm:mainestimate} was proved by Widmer \cite[Theorem 2.1]{Widmer:WeakAdmiss}.

\subsection{Applications I}
\label{subsec:DiophApprox}

Theorem \ref{thm:mainestimate} has some interesting applications in multiplicative Diophantine approximation. Let $\bs{Q}\in(0,+\infty)^{\mc{N}}$ and let $X:=\prod_{j=1}^{\mc{N}}[-Q_{j},Q_{j}]$. We set
\begin{equation}
S_{\bs{L}}(\bs{Q}):=\sum_{\substack{\bs{q}\in X\cap\mb{Z}^{\mc{N}}\setminus\{\bs{0}\}}}\prod_{i=1}^{\mc{M}} \|L_{i}(\bs{q})\|^{-1}. \nonumber
\end{equation}
The function $S_{\bs{L}}(\bs{Q})$ is of major importance in several branches of Diophantine Approximation and Geometry of Numbers. For instance, Kuipers and Niederreiter use it to control the discrepancy of the sequence $\{q\bs{L}\}_{q\in\mb{Z}}$ via the Erd\H{o}s-Turan inequality \cite[p.122,129,131]{KuipersNiederr:UnifDistr}. Hardy and Littlewood use it to count the number of lattice points contained in certain polygons \cite{HarLi:SomeProblems1}\cite{HarLi:SomeProblems2}. Beresnevich, Haynes, and Velani, work out very precise estimates for $S_{\bs{L}}(\bs{Q})$ in the one dimensional inhomogeneous setting \cite{Beresnevich:Somsofandmult}. Huang and Liu show how estimates of $S_{\bs{L}}(\bs{Q})$ can be used to solve the convergence case of the generalised Baker-Schmidt problem for simultaneous approximation on certain affine subspaces \cite{HuangLiu:Simultaneous}.  

In this section, we focus on a question raised by L\^e and Vaaler \cite{LeVaaler:Sumsof}. L\^e and Vaaler show that for $Q:=(Q_{1}\dotsm Q_{\mc{N}})^{1/\mc{N}}\geq 1$ it holds
\begin{equation}
\label{eq:intro2}
S_{\bs{L}}(\bs{Q})\gg_{\mc{M},\mc{N}}Q^{\mc{N}}(\log Q)^{\mc{M}},
\end{equation}
independently of the choice of the matrix $\bs{L}$ \cite[Corollary 1.2]{LeVaaler:Sumsof}. They also ask whether the estimate in (\ref{eq:intro2}) is sharp, i.e., whether there exists a matrix $\bs{L}$ such that
\begin{equation}
\label{eq:intro3}
S_{\bs{L}}(\bs{Q})\ll Q^{\mc{N}}(\log Q)^{\mc{M}}.
\end{equation}
L\^e and Vaaler themselves show that (\ref{eq:intro3}) holds true whenever the matrix $\bs{L}$ is multiplicatively badly approximable \cite[Theorem 2.1]{LeVaaler:Sumsof}. However, multiplicatively badly approximable matrices seem unlikely to exist for $\mc{M}+\mc{N}\geq 3$, since each of them would provide a counterexample to the Littlewood conjecture. In the present section we prove a general estimate for the function $S_{\bs{L}}(\bs{Q})$, when $Q_{1}=Q_{2}=\dotsb =Q_{\mc{N}}$. This estimate shows that, in the special case $Q_{1}=Q_{2}=\dotsb =Q_{\mc{N}}$, L\^e and Vaaler's hypothesis can be significantly weakened.
Let us first recap some definitions (see \cite{Bugeaud:Multiplicative}\cite{Schmidt:BadlyApprox} for a deeper insight).
\begin{definition}
\label{def:multaddbadapp}
Let $\phi:[1,+\infty)\to (0,1]$ be non-increasing. We say that a matrix $\bs{L}\in\mb{R}^{\mc{M}\times\mc{N}}$ is $\phi$-additively badly approximable if
\begin{equation}
\label{eq:condition1add}
|\bs{q}|_{\infty}^{\mc{N}}\left(\max_{i=1}^{\mc{M}}\|L_{i}(\bs{q})\|\right)^{\mc{M}}\geq\phi(|\bs{q}|_{\infty})
\end{equation}
for all $\bs{q}\in\mb{Z}^{\mc{N}}\setminus\{\bs{0}\}$. We say that $\bs{L}$ is $\phi$-multiplicatively badly approximable if
\begin{equation}
\label{eq:condition1mult}
\prod_{j=1}^{\mc{N}}\max\left\{1,\left|q_{j}\right|\right\}\prod_{i=1}^{\mc{M}}\|L_{i}(\bs{q})\|\geq\phi\left(\prod_{j=1}^{\mc{N}}\max\left\{1,\left|q_{j}\right|\right\}^{\frac{1}{\mc{N}}}\right)
\end{equation}
for all $\bs{q}\in\mb{Z}^{\mc{N}}\setminus\{\bs{0}\}$.
If $\phi$ can be chosen constant in either case, we say that $\bs{L}$ is additively or multiplicatively badly approximable.
\end{definition}

The additive and multiplicative conditions are very different. Schmidt \cite{Schmidt:BadlyApprox} showed that the set of additively badly approximable matrices in $\mb{R}^{\mc{M}\times\mc{N}}$ has full Hausdorff dimension. On the other hand, as mentioned above, multiplicatively badly approximable matrices are unlikely to exist for $\mc{M}+\mc{N}\geq 3$.

We introduce a new condition, which is hybrid between (\ref{eq:condition1add}) and (\ref{eq:condition1mult}).

\begin{definition}
Let $\phi$ be as in Definition \ref{def:multaddbadapp}. We say that a matrix $\bs{L}\in\mb{R}^{\mc{M}\times\mc{N}}$ is $\phi$-semi multiplicatively badly approximable if
\begin{equation}
\label{eq:condition1}
|\bs{q}|_{\infty}^{\mc{N}}\prod_{i=1}^{\mc{M}}\|L_{i}(\bs{q})\|\geq\phi(|\bs{q}|_{\infty})
\end{equation}
for all $\bs{q}\in\mb{Z}^{\mc{N}}\setminus\{\bs{0}\}$. If the function $\phi$ can be chosen constant, we say that $\bs{L}$ is semi-multiplicatively badly approximable.
\end{definition}

Note that (\ref{eq:condition1mult})$\Rightarrow$(\ref{eq:condition1})$\Rightarrow$(\ref{eq:condition1add}).
Now, under the additional hypothesis $Q_{1}=\dotsb =Q_{\mc{N}}$, we have the following estimate for $S_{\bs{L}}(\bs{Q})$. 

\begin{cor}
\label{cor:cor2}
Let $\bs{L}\in\mb{R}^{\mc{M}\times\mc{N}}$ be a $\phi$-semi multiplicatively badly approximable matrix. Then, for $Q\geq 2$ we have
\begin{equation}
\sum_{\substack{\bs{q}\in[-Q,Q]^{\mc{N}}\\ \cap\ \mb{Z}^{\mc{N}}\setminus\{\bs{0}\}}}\prod_{i=1}^{\mc{M}} \|L_{i}(\bs{q})\|^{-1}\ll_{\mc{M},\mc{N}}Q^{\mc{N}}\log\left(\frac{Q^{\mc{N}}}{\phi(Q)}\right)^{\mc{M}}+\frac{Q^{\mc{N}}}{\phi(Q)}\log\left(\frac{Q^{\mc{N}}}{\phi(Q)}\right)^{\mc{M}-1}.\nonumber
\end{equation}
\end{cor}

Corollary \ref{cor:cor2} immediately implies that, in the special case $Q_{1}=\dotsb =Q_{\mc{N}}$, all matrices $\bs{L}$ that are $\phi$-semi multiplicatively badly approximable with $\phi(x)\gg_{\mc{M},\mc{N}} (\log x)^{-1}$ satisfy (\ref{eq:intro3}).

The case $\mc{M}=2$, $\mc{N}=1$ of Corollary \ref{cor:cor2} was proved by Widmer \cite{Widmer:AsymptoticDiophantine}, whereas the analogous result for $\mc{M}=1$ was proved by the author \cite{Fregoli:Sumsof}, by using Widmer's $\mc{M}=1$ case of Theorem \ref{thm:mainestimate} \cite[Theorem 2.1]{Widmer:WeakAdmiss}. Huang and Liu addressed the general case \cite[Theorem 6]{HuangLiu:Simultaneous} by using the well-known gap principle \cite[Proof of Lemma 3.3 p.123]{KuipersNiederr:UnifDistr}. However, Huang and Liu's Theorem 6 contains an extra factor of $1/\phi(Q)$ in the error term, which we could get rid of in Corollary \ref{cor:cor2}.

Unfortunately, the existence of $\phi$-semi multiplicatively badly approximable matrices with $\phi(x)\gg_{\mc{M},\mc{N}}(\log x)^{-1}$ has not yet been established, except for $\mc{M}=\mc{N}=1$. Despite this, at least in dimension $2$, there is some heuristic evidence for their existence. For $\mc{M}=2$, $\mc{N}=1$ Badziahin showed that condition (\ref{eq:condition1mult}), with $\phi(x)=c_{\bs{L}}(\log x\log\log x)^{-1}$ ($c_{\bs{L}}>0$ sufficiently small), holds true for a set of vectors of full Hausdorff dimension \cite{Badziahin:OnMultiplicatively}.
It follows from Corollary \ref{cor:cor2} that for $\mc{M}=2$, $\mc{N}=1$ the set of matrices $\bs{L}$ such that $S_{\bs{L}}(\bs{Q})\ll_{\bs{L}}Q(\log Q)^{2}\log\log Q$ has full Hausdorff dimension. Badziahin and Velani also conjectured that the set of $2\times 1$ $\phi$-multiplicatively badly approximable matrices, with $\phi(x)=c_{\bs{L}}(\log x)^{-1}$ ($c_{\bs{L}}>0$ sufficiently small), has full Hausdorff dimension \cite[Conjecture 1]{Badziahin:OnMultiplicatively}. To the best of our knowledge, nothing is known in higher dimension.

\subsection{Applications II} 

Let $\psi:[1,+\infty)\to(0,1]$. We consider the set
$$\mathscr{S}_{\mc{N}}(\psi):=\left\{\bs{x}\in\mb{R}^{\mc{N}}:\exists\ \mbox{i.m. }q\in\mb{N}\ \mbox{such that }\max_{i=1}^{\mc{N}}\|qx_{i}\|<\psi(q)\right\},$$
where i.m. stands for infinitely many. The set $\mathscr{S}_{\mc{N}}$ is said to be the set of simultaneously $\psi$-approximable points. A well-known theorem of Khintchine \cite{Khintchine:Zurmetrischen} relates the Lebesgue measure of the set $\mathscr{S}_{\mc{N}}(\psi)$ to the convergence of the sum $\sum_{q=1}^{+\infty}\psi(q)^{\mc{N}}$.

Khintchine showed that if $\psi$ is non-increasing and $\sum_{q=1}^{+\infty}\psi(q)^{\mc{N}}$ converges, we have
\\ $\Vol\left(\mathscr{S}_{\mc{N}}(\psi)\right)=0$, whereas if $\psi$ is non-increasing and $\sum_{q=1}^{+\infty}\psi(q)^{n}$ diverges, we have \\ $\Vol\left(\mathscr{S}_{\mc{N}}(\psi)\right)=+\infty$.
It is well known that this theorem becomes false when we restrict to certain submanifolds of $\mb{R}^{\mc{N}}$, such as proper rational affine subspaces. This leads naturally to the following definition.

\begin{definition}
\label{def:Khintchine}
Let $\mc{M}\subset\mb{R}^{\mc{N}}$ be a submanifold of dimension $d$. We say that $\mc{M}$ is of Khintchine type for convergence if for all non-increasing functions $\psi:[1,+\infty)\to(0,1]$ we have
$$\sum_{q=1}^{+\infty}\psi(q)^{\mc{N}}<+\infty\Rightarrow\Vol_{d}\left(\mathscr{S}_{\mc{N}}(\psi)\cap\mc{M}\right)=0.$$
We say that $\mc{M}$ is of Khintchine type for divergence if for all non-increasing functions $\psi:[1,+\infty)\to(0,1]$ we have
$$\sum_{q=1}^{+\infty}\psi(q)^{\mc{N}}=+\infty\Rightarrow\Vol_{d}\left(\mathscr{S}_{\mc{N}}(\psi)\cap\mc{M}\right)=+\infty.$$
If both conditions hold, we simply say that $\mc{M}$ is of Khintchine type.
\end{definition}

We recall that there is also a notion of strong Khintchine type submanifold in $\mb{R}^{\mc{N}}$, i.e., a submanifold for which Definition \ref{def:Khintchine} holds without the assumption that the function $\psi$ is non-increasing.

It has been shown that many non-degenerate submanifolds (i.e., those that in some sense deviate from a hyperplane at each point) are of strong Khintchine type for convergence \cite{Huang:Diophantineapprxontheparabola},\cite{Simmons:Somemanifolds}. It seems natural then, to ask whether the non-degeneracy condition is necessary for a submanifold to be of (strong) Khintchine type. The answer to this question is no, and indeed it turns out that even some proper affine subspaces of $\mb{R}^{\mc{N}}$ are of strong Khintchine type \cite{Kovalevskaya:Ontheexact},\cite{Ramirez:Khintchine}. So, what makes an affine subspace of (strong) Khintchine type? Since each affine subspace is defined by a real matrix, it appears interesting to try and establish a link between the Diophantine type of this matrix and the properties of the subspace in terms of the validity of the Khintchine Theorem.
In a very recent paper \cite{HuangLiu:Simultaneous} Huang and Liu made some progress in this direction.

\begin{definition}
Let $\bs{L}\in\mb{R}^{\mc{M}\times\mc{N}}$. We set
$$\omega_{m}(\bs{L}):=\sup\left\{\gamma\in\mb{R}:\prod_{j=1}^{\mc{M}}\left\|\bs{L}_{j}(\bs{q})\right\|\leq|\bs{q}|_{\infty}^{-\gamma}\ \mbox{has i.m. solutions }\bs{q}\in\mb{Z}^{\mc{N}}\setminus\{\bs{0}\}\right\}.$$
We call $\omega_{m}(\bs{L})$ the multiplicative exponent of the matrix $\bs{L}$. 
\end{definition}

\noindent Observe that for $\phi_{\gamma}(x):=x^{-\gamma}$ ($\gamma\in\mb{R}$) we have 
\begin{equation}
\omega_{m}(\bs{L})=\inf\{\gamma\in\mb{R}:\exists c>0\ \mbox{such that }\bs{L}\ \mbox{is }c\phi_{(\gamma-\mc{N})}\mbox{-semi multiplicatively badly approximable}\}\nonumber.
\end{equation} 

Let $d\geq 1$ be an integer, and let $\bs{A}\in\mb{R}^{d\times(\mc{N}-d)}$. Let also $\bs{\alpha}_{0}\in\mb{R}^{\mc{N}-d}$. We define
\begin{equation}
\label{eq:tilde}
\tilde{\bs{A}}:=\binom{\bs{\alpha}_{0}}{\bs{A}}\in\mb{R}^{(d+1)\times(\mc{N}-d)}\quad\mbox{and}\quad\tilde{\bs{x}}:=(1,\bs{x})\in\mb{R}^{d+1}\ \mbox{for }\bs{x}\in\mb{R}^{d}.\end{equation}
Then, we consider the following submanifold of $\mb{R}^{\mc{N}}$.
$$\mathscr{H}:=\{(\bs{x},\tilde{\bs{x}}\tilde{\bs{A}}):\bs{x}\in[0,1]^{d}\}.$$

Huang and Liu \cite[Theorem 1]{HuangLiu:Simultaneous} proved that if $\omega_{m}\left(\tilde{\bs{A}}\right)<\mc{N}(d+1)$, the submanifold $\mathscr{H}$ is of Khintchine type for convergence, whereas if $\omega_{m}(\bs{A})<\mc{N}d$, the submanifold $\mathscr{H}$ is of strong Khintchine type for convergence.
We recall that for all $\omega_{0}\geq \mc{N}-d$ there always exist matrices $\bs{A}\in\mb{R}^{d\times(\mc{N}-d)}$ such that $\omega_{m}(\bs{A})=\omega_{0}$ \cite[Theorem 1]{Bugeaud:Multiplicative}. More precisely, we have
$$\textup{dim}\left\{\bs{A}\in\mb{R}^{d\times(\mc{N}-d)}:\omega_{m}(\bs{A})=\omega_{0}\right\}=d(\mc{N}-d)-1+\frac{2}{1+\omega_{0}/(\mc{N}-d)},$$
where $\textup{dim}$ denotes the Hausdorff dimension. An analogous formula holds for $\tilde{\bs{A}}$. It can also be shown that for all $\varepsilon>0$ the set of matrices $\bs{A}\in\mb{R}^{d\times(\mc{N}-d)}$ such that $\omega(\bs{A})\leq \mc{N}-d+\varepsilon$ has actually full Lebesgue measure. It follows that Huang and Liu's theorem holds for generic matrices $\bs{A}$ and $\tilde{\bs{A}}$.

One could ask if anything can be said about the limit cases, i.e., $\omega(\tilde{\bs{A}})=\mc{N}(d+1)$ and $\omega(\bs{A})=\mc{N}d$ . We show that that, up to a logarithmic factor, these cases yield Khintchine type subspaces.

\begin{definition}
Let $\bs{L}\in\mb{R}^{\mc{M}\times\mc{N}}$ and let $\omega_{0}\in\mb{R}$. We set
\begin{multline}\omega'_{m}(\bs{L},\omega_{0}):= \\
\sup\left\{\gamma\in\mb{R}:\prod_{j=1}^{\mc{M}}\left\|\bs{L}_{j}\bs{q}\right\|\leq|\bs{q}|_{\infty}^{-\omega_{0}}\log(|\bs{q}|_{\infty})^{-\gamma}\ \mbox{has i.m. solutions }\bs{q}\in\mb{Z}^{\mc{N}}\setminus[-1,1]^{\mc{N}}\right\}.\nonumber
\end{multline}
We call $\omega_{m}(\bs{L},\omega_{0})$ the multiplicative logarithmic exponent of the matrix $\bs{L}$ at $\omega_{0}$.
\end{definition}

\begin{cor}
\label{cor:HLrefined}
Let $\bs{A}$ and $\tilde{\bs{A}}$ be as above. Then,
\begin{itemize}
\item[$i)$] if $\omega_{m}\left(\tilde{\bs{A}}\right)=\mc{N}(d+1)$ and $\omega_{m}'\left(\tilde{\bs{A}},\mc{N}(d+1)\right)< 1-2(d+1)$, the submanifold $\mathscr{H}$ is of Khintchine type for convergence;\vspace{2mm}
\item[$ii)$] if $\omega_{m}(\bs{A})=\mc{N}d$ and $\omega_{m}'(\bs{A},\mc{N}d)< 1-2d$, the submanifold $\mathscr{H}$ is of strong Khintchine type for convergence.
\end{itemize}
\end{cor}

\noindent Unfortunately, not much is known about the existence of matrices with prescribed multiplicative logarithmic order. However, their existence is established in the additive setting \cite{Beresnevich:Setsofexactlogorder}. We can therefore say something about the case $d=1$. From \cite[Theorem 1]{Beresnevich:Setsofexactlogorder} we can easily deduce that if $d=1$, there always exist matrices $\bs{A}\in\mb{R}^{1\times\mc{N}}$ such that $\omega_{m}'(\bs{A},\mc{N})=\omega_{1}$ for any given $\omega_{1}\in\mb{R}$ (this implies $\omega_{m}(\bs{A})=\mc{N}$). More precisely, we have
$$\textup{dim}\left\{\bs{A}\in\mb{R}^{1\times(\mc{N}-1)}:\omega_{m}'(\bs{A},\mc{N})=\omega_{1}\right\}=\mc{N}-2+\frac{\mc{N}}{1+\mc{N}},$$
independently of the choice of $\omega_{1}$ (here $\textup{dim}$ denotes the Hausdorff dimension). It follows from Corollary \ref{cor:HLrefined} that there exist strong Khintchine type lines in $\mb{R}^{\mc{N}}$ with exponent $\omega_{m}=\mc{N}$, improving on \cite[Theorem 1]{HuangLiu:Simultaneous}.

Now, \cite[Theorem 1]{HuangLiu:Simultaneous} follows in turn from \cite[Theorems 2 and 3]{HuangLiu:Simultaneous}. These results establish some Khintchine type conditions for the submanifold $\mathscr{H}$ with respect to general $s$-dimensional Hausdorff measures (i.e., $s$ need not coincide with the dimension of the submanifold). The problem of establishing Khintchine type conditions with respect to general Hausdorff measures is widely known as the generalised Baker-Schmidt problem. Huang and Liu prove that such conditions hold for the convergence case, when $\omega(\tilde{\bs{A}})<(\mc{N}+1)d$ or $\omega(\bs{A})<\mc{N}d$. We show that \cite[Theorems 2 and 3]{HuangLiu:Simultaneous} can be refined to include the limit cases.

\begin{prop}
\label{thm:thm3refined}
Let $\bs{A}\in\mb{R}^{d\times(\mc{N}-d)}$ and let $\bs{\alpha}_{0}\in\mb{R}^{\mc{N}-d}$. Let $\tilde{\bs{A}}$ be the matrix defined in (\ref{eq:tilde}), and let $s\in[0,+\infty)$. Assume that $\tilde{\bs{A}}$ is $\tilde{\phi}$-semi multiplicatively badly approximable, where $\tilde{\phi}:[1,+\infty)\to(0,1]$ is a non-increasing function with the following properties:\vspace{2mm}
\begin{itemize}
\item[$i)$] $\tilde{\phi}(\lambda x)\gg_{\lambda}\tilde{\phi}(x)$ for all $\lambda\gg 1$;\vspace{2mm}
\item[$ii)$] $x^{-\gamma}\ll\tilde{\phi}(x)\ll 1/\log(x)$ for some $\gamma>0$;\vspace{2mm}
\item[$iii)$] there exists a non-increasing function $\hat{\psi}:[1,+\infty)\to(0,1]$ such that\vspace{2mm}
 \begin{itemize}
 \item[$iiia)$] $\sum_{q=1}^{+\infty}\hat{\psi}(q)^{\mc{N}-d+s}q^{d-s}<+\infty$;\vspace{2mm}
 \item[$iiib)$] $\tilde{\phi}\left(1/\hat{\psi}(x)\right)\hat{\psi}(x)^{\mc{N}-d}\gg_{\mc{N},d,s}\log(x)^{d}/x^{d+1}$.\vspace{2mm}
 \end{itemize}
\end{itemize}
Then, for all non-increasing approximating functions $\psi:[1,+\infty)\to(0,1]$ such that \\
$\sum_{q=1}^{+\infty}\psi(q)^{\mc{N}-d+s}q^{d-s}<+\infty$, we have $\textup{Vol}_{s}\left(\mathscr{S}_{\mc{N}}(\psi)\cap\mathscr{H}\right)=0$.
\end{prop}

Note that when $\hat{\psi}(x)$ is of the form $\hat{\psi}(x)=x^{-\gamma'}$, with $\gamma'>0$, condition $iii)$ implies $\omega_{m}\left(\tilde{\bs{A}}\right)\leq(d+1)(d-s+1)/(\mc{N}-d+s)$, i.e., the hypothesis in Huang and Liu's theorem along with the limit case.

\begin{prop}
\label{thm:thm2refined}
Let $s\in[0,+\infty)$ and let $\bs{A}\in\mb{R}^{d\times(\mc{N}-d)}$ be a $\phi$-semi multiplicatively badly approximable matrix, where $\phi:[1,+\infty)\to(0,1]$ is a non-increasing function with the following properties:\vspace{2mm}
\begin{itemize}
\item[$i)$] $\phi(\lambda x)\gg_{\lambda}\phi(x)$ for all $\lambda\gg 1$;\vspace{2mm}
\item[$ii)$] $x^{-\gamma}\ll\phi(x)\ll1/\log(x)$ for some $\gamma>0$;\vspace{2mm}
\item[$iii)$] there exists a function $\hat{\psi}:[1,+\infty)\to(0,1]$ such that\vspace{2mm}
 \begin{itemize}
 \item[$iiia)$] $\sum_{q=1}^{+\infty}\hat{\psi}(q)^{\mc{N}-d+s}q^{d-s}<+\infty$;\vspace{2mm}
 \item[$iiib)$] $\phi\left(1/\hat{\psi}(x)\right)\hat{\psi}(x)^{\mc{N}-d}\gg_{\mc{N},d,s}\log(x)^{d-1}/x^{d}$.\vspace{2mm}
 \end{itemize}
\end{itemize}
Then, for all approximating functions $\psi:[1,+\infty)\to(0,1]$ such that $\sum_{q=1}^{+\infty}\psi(q)^{\mc{N}-d+s}q^{d-s}<+\infty$, we have $\textup{Vol}_{s}\left(\mathscr{S}_{\mc{N}}(\psi)\cap\mathscr{H}\right)=0$.
\end{prop}

\noindent With Propositions \ref{thm:thm3refined} and \ref{thm:thm2refined} at hand, the proof of Corollary \ref{cor:HLrefined} is straightforward. We sketch it below.

\begin{proof}
Let $\log^{*}(x):=\max\{1,\log(x)\}$ for all $x\in(0,+\infty)$. The proof follows from taking $\tilde{\phi}(x)=\tilde{c}x^{\mc{N}-d-\omega_{m}\left(\tilde{\bs{A}}\right)}\log^{*}(x)^{-\omega_{m}'\left(\tilde{\bs{A}}\right)-\varepsilon}$ ($\tilde{c}>0$) and $\phi(x)=cx^{\mc{N}-d-\omega_{m}(\bs{A})}\log^{*}(x)^{-\omega_{m}'(\bs{A})-\varepsilon}$ ($c>0$), and applying the case $s=d$ of Propositions \ref{thm:thm2refined} and \ref{thm:thm3refined}, with $\hat{\psi}(x):=x^{-1/\mc{N}}\log^{*}(x)^{(-1-\varepsilon)/\mc{N}}$.
\end{proof}

Note that we intentionally chose not to specify the function $\hat{\psi}$ in Propositions \ref{thm:thm3refined} and \ref{thm:thm2refined}, since these results could be used to derive even finer Diophantine conditions on subspaces, involving, e.g, iterated logarithms.

\section{Proof of Theorem \ref{thm:mainestimate}}

From now on, we fix an $o$-minimal structure $\mathfrak{S}$ extending $\mb{R}_{\exp}$, and we say that a set $S\subset\mb{R}^{n}$ is definable if it is definable in $\mathfrak{S}$.
We fix the parameters $\bs{\tau}$ and $\bs{\eta}(\bs{\tau})=(\varepsilon,T,\bs{Q})\in(0,+\infty)^{2+N}$ such that $W_{\bs{\tau}}\subset Z_{\bs{\eta}(\bs{\tau})}$. For simplicity, we set $W:=W_{\bs{\tau}}$ and $Z:=Z_{\bs{\eta}(\bs{\tau})}$. We also write $H$ for $H_{\varepsilon,T}$ and $R$ for $R_{\bs{Q}}$.

To prove our estimate, we partition the set $Z$ and we consider the induced partition on $W$. We then count the lattice points contained in each subset of this partition. Let
$$H_{+}:=H\cap\left\{\underline{\bs{x}}\in V_{\bs{m}}:\bs{x}_{i}\neq \bs{0}\ i=1,\dotsc,M\right\}$$
and let $Z_{+}:=H_{+}\times R$. Let also
$$H^{i}:=H\cap\{\bs{x}_{i}=\bs{0}\}$$
and $Z^{i}:=H^{i}\times R$ for $i=1,\dotsc,M$. We set $W_{+}:=W\cap Z_{+}$ and $W^{i}:=W\cap Z^{i}$ for $i=1,\dotsc,M$. Then, we have
\begin{equation}
W=W_{+}\cup\bigcup_{i=1}^{M}W^{i}.\nonumber
\end{equation}
Hence,
\begin{equation}
\left|\#(\Lambda\cap W)-\frac{\Vol\,W}{\det\Lambda}\right|\leq\left|\#(\Lambda\cap W_{+})-\frac{\Vol\,W_{+}}{\det\Lambda}\right|+\sum_{i=1}^{M}\#\left(\Lambda\cap W^{i}\right),\nonumber
\end{equation}

To decompose the sets $H_{+}$ and $W_{+}$ we use the following crucial decomposition result.

\begin{prop}
\label{prop:partition}
Assume that $T/\varepsilon>e$ (where $e=2.71828\dots$ is the base of the natural logarithm). Then, there exists a partition of the set $H_{+}$ of the form $H_{+}=\bigcup_{k\in\mc{K}}\! X_{k}$, and there exists a collection of linear maps $\left\{\varphi_{k}\right\}_{k\in{\mc{K}}}$, defined on the space $V_{\bs{m}}$, such that\vspace{2mm}
\begin{itemize}
\item[$i)$] $\#\mc{K}\ll_{\bs{m},\bs{\beta}}\log\left(T/\varepsilon\right)^{M-1}$;\vspace{2mm}
\item[$ii)$] each of the sets $X_{k}$ for $k\in\mc{K}$ is definable;\vspace{2mm}
\item[$iii)$] the maps $\varphi_{k}$ for $k\in\mc{K}$ are defined by $\varphi_{k}(\underline{\bs{x}})_{i}=\exp\left(a^{k}_{i}-c\right)\bs{x}_{i}$ for $i=1,\dotsc,M$, where $c\in\mb{R}$ is a constant only depending on $\bs{m}$ and the coefficients $a^{k}_{i}\in\mb{R}$ satisfy\vspace{2mm}
\begin{itemize}
\item[$iiia)$] $\exp\left(a^{k}_{i}-c\right)\gg_{\bs{m},\bs{\beta}}\varepsilon/T$ for $i=1,\dotsc,M$;\vspace{2mm}
\item[$iiib)$] $\sum_{i=1}^{M}\beta_{i}a^{k}_{i}=0$;\vspace{2mm}
\end{itemize}
\vspace{2mm}
\item[$iv)$] $\varphi_{k}\left(X_{k}\right)\subset\left\{|\bs{x}_{i}|_{2}\leq\varepsilon\ i=1,\dotsc,M\right\}$ for $k\in\mc{K}$.\vspace{2mm}
\end{itemize}
\end{prop}

We prove Proposition \ref{prop:partition} in Section \ref{sec:partitionproof}. The following corollary is an immediate consequence of Proposition \ref{prop:partition}.

\begin{cor}
\label{cor:partition}
Let $\hat{X}_{k}:=X_{k}\times R\subset V$ and let $\hat{\varphi}_{k}:=(\varphi_{k},\textup{id}):V\to V$ for all $k\in\mc{K}$. Then,\vspace{2mm}
\begin{itemize}
\item[$i)$] $Z_{+}=\bigcup_{k\in\mc{K}}\hat{X}_{k}$ is a partition of the set $Z_{+}$;\vspace{2mm}
\item[$ii)$] each of the sets $\hat{X}_{k}$ for $k\in\mc{K}$ is definable;\vspace{2mm}
\item[$iii)$] $\hat{\varphi}_{k}(\hat{X}_{k})\subset\left\{|\bs{x}_{i}|_{2}\leq\varepsilon\ i=1,\dotsc,M\right\}\times R$ for $k\in\mc{K}$.\vspace{2mm}
\end{itemize}
\end{cor}
Corollary \ref{cor:partition} yields the following partition of the set $W_{+}$.
$$W_{+}=\bigcup_{k\in\mc{K}}W\cap\hat{X}_{k}.$$
Hence, we can write
\begin{align}
\label{eq:estimate1} 
 & \left|\#\left(\Lambda\cap W\right)-\frac{\Vol\,W}{\det\Lambda}\right|\leq\left|\#(\Lambda\cap W_{+})-\frac{\Vol\,W_{+}}{\det\Lambda}\right|+\sum_{i=1}^{M}\#(\Lambda\cap W^{i})\nonumber\\
 & \leq\sum_{k\in \mc{K}}\left|\#\left(\Lambda\cap W\cap\hat{X}_{k}\right)-\frac{\Vol\,(W\cap\hat{X}_{k})}{\det\Lambda}\right|+\sum_{i=1}^{M}\#(\Lambda\cap W^{i})\nonumber \\
 & =\sum_{k\in\mc{K}}\left|\#\left(\hat{\varphi}_{k}(\Lambda)\cap \hat{\varphi}_{k}\left(W\cap\hat{X}_{k}\right)\right)-\frac{\Vol\,\hat{\varphi}_{k}\left(W\cap\hat{X}_{k}\right)}{\det\hat{\varphi}_{k}(\Lambda)}\right|+\sum_{i=1}^{M}\#(\Lambda\cap W^{i}).
\end{align}

\begin{lem}
\label{lem:LambdacapC}
Let $c:=\dim C$. Then, for $i=1,\dotsc,M$ we have
\begin{equation}
\#\left(\Lambda\cap W^{i}\right)\ll_{\bs{m},\bs{n}}1+\left(\frac{\diam(W\cap C)}{\lambda_{1}(\Lambda\cap C)}\right)^{c}.\nonumber
\end{equation}
\end{lem}

\begin{proof}
By weak admissibility, we have $\Lambda\cap W_{i}\subset\Lambda\cap W\cap C$. Therefore, it is enough to estimate $\#(\Lambda\cap W\cap C)$. Now, $\Lambda\cap C$ is either $\{(\bs{0},\bs{0})\}$ or a full rank lattice in some subspace $C'\subset C$ with $\textup{dim}(C')=c'>0$. To prove the claim, it suffices to show that for any bounded set $S\subset\mb{R}^{n}$ and any full rank lattice $\Gamma\subset\mb{R}^{n}$ we have
\begin{equation}
\label{eq:convexestimate}
\#(\Gamma\cap S)\ll_{n}1+\left(\frac{\diam S}{\lambda_{1}(\Gamma)}\right)^{n}.
\end{equation}
This follows easily from \cite[Lemmas 2.1 and 2.2]{Widmer:CountingLattice}.
Applying (\ref{eq:convexestimate}) to $\left(W\cap C'\right)\cap\left(\Lambda\cap C'\right)$ and noting that $c'\leq\mc{M}+\mc{N}-1$ and $\lambda_{1}(\Lambda\cap C)=\lambda_{1}(\Lambda\cap C')$ yields
\begin{multline}
\#\left(\Lambda\cap W^{i}\right)\leq\#\left((W\cap C')\cap(\Lambda\cap C')\right)\ll_{c'}1+\left(\frac{\diam(W\cap C)}{\lambda_{1}(\Lambda\cap C)}\right)^{c'} \\
\ll_{\bs{m},\bs{n}} 1+\left(\frac{\diam(W\cap C)}{\lambda_{1}(\Lambda\cap C)}\right)^{c}\nonumber.
\end{multline}
Note that in the last inequality we can replace $c'$ by a bigger integer, due to the definition of the constant in (\ref{eq:convexestimate}) (see again \cite[Lemmas 2.1 and 2.2]{Widmer:CountingLattice}).
\end{proof}

We are left to estimate the quantity $\#\left(\hat{\varphi}_{k}(\Lambda)\cap \hat{\varphi}_{k}\left(W\cap\hat{X}_{k}\right)\right)-\Vol\,\hat{\varphi}_{k}\left(W\cap\hat{X}_{k}\right)/\det\hat{\varphi}_{k}(\Lambda)$ for $k\in\mc{K}$. By Corollary \ref{cor:partition}, we know that
\begin{equation}
\label{eq:27}
\hat{\varphi}_{k}\left(W\cap\hat{X}_{k}\right)\subset\hat{\varphi}_{k}\left(\hat{X}_{k}\right)\subset\left\{|\bs{x}_{i}|_{2}\leq\varepsilon\ i=1,\dotsc,M\right\}\times R.
\end{equation}
To make the counting more effective, we reshape the set on the right-hand side of (\ref{eq:27}) into a ball-like shaped set. Let $\omega_{1}:V\to V$ be the map
$$\omega_{1}(\underline{\bs{x}},\underline{\bs{y}}):=\left(\underline{\bs{x}},\frac{Q}{Q_{1}}\bs{y}_{1},\dotsc,\frac{Q}{Q_{N}}\bs{y}_{N}\right),$$
and let $\omega_{2}:V\to V$ be the map
$$\left(\underline{\bs{x}},\underline{\bs{y}}\right)\mapsto\left(\theta\underline{\bs{x}},\theta^{-\frac{\mc{B}}{\mc{C}}}\underline{\bs{y}}\right),$$
where
$$\theta:=\frac{\left(\varepsilon^{\mc{B}} Q^{\mc{C}}\right)^{\frac{1}{\mc{B}+\mc{C}}}}{\varepsilon}.$$   
Then, we have
\begin{multline}
\label{eq:28}
\omega_{2}\circ\omega_{1}(\left\{|\bs{x}_{i}|_{2}\leq\varepsilon\ i=1,\dotsc,M\right\}\times R) \\
=\left\{|\bs{x}_{i}|_{2}\leq\left(\varepsilon^{\mc{B}}Q^{\mc{C}}\right)^{\frac{1}{\mc{B}+\mc{C}}}\ i=1,\dotsc,M\right\}\times\left\{|\bs{y}_{j}|_{2}\leq\left(\varepsilon^{\mc{B}}Q^{\mc{C}}\right)^{\frac{1}{\mc{B}+\mc{C}}}\ j=1,\dotsc,N\right\}.
\end{multline}

Now, to complete the estimate we use the following general counting result \cite[Theorem 1.3]{Widmer:CountingLattice}, which we state for a vector space of the form $V_{\bs{l}}$ and a definable family.
\begin{theorem}[Barroero-Widmer]
\label{thm:latticeest}
Let $\bs{l}\in\mb{N}^{L}$ and let $\mc{L}:=\sum_{h=1}^{L}l_{h}$. Let also $t\in\mb{N}$. Consider a full rank lattice $\Lambda\subset V_{\bs{l}}$ and a definable family $\mc{W}'\subseteq V_{\bs{l}}\times\mb{R}^{t}$. Suppose that each fibre $W_{\bs{\tau}}'$ of $\mc{W}'$ is bounded. Then, there exists a constant $c_{\mc{W}'}\in\mb{R}$, only depending on $\mc{W}'$, such that
$$\left|\#(\Lambda\cap W_{\bs{\tau}}')-\frac{\Vol\,W_{\bs{\tau}}'}{\det\Lambda}\right|\leq c_{\mc{W}'}\sum_{s=0}^{\mc{L}-1}\frac{V_{s}(W_{\bs{\tau}}')}{\lambda_{1}\dotsm\lambda_{s}},$$
where $V_{s}\left(W_{\bs{\tau}}'\right)$ is the sum of the $s$-dimensional volumes of the orthogonal projections of $W_{\bs{\tau}}'$ onto every $s$-dimensional coordinate subspace of $V_{\bs{l}}$, and $\lambda_{s}$ is the $s$-th successive minimum of the lattice $\Lambda$ with respect to the Euclidean unit ball. By convention, $V_{0}(W_{\bs{\tau}}')=\lambda_{0}=1$.
\end{theorem}

\noindent We fix $k\in\mc{K}$, and we apply Theorem \ref{thm:latticeest} to the family
$$\mc{S}_{k}:=\left\{(\omega_{2}\circ\omega_{1}\circ\hat{\varphi}_{k}(\underline{\bs{v}}),\bs{\tau}):(\underline{\bs{v}},\bs{\tau})\in \mc{W}\cap\left(\hat{X}_{k}\times\mb{R}^{t}\right)\right\}\subset V\times\mb{R}^{t}.$$
This family is definable in view of Definition \ref{def:definablestruc} and part $ii)$ of Corollary \ref{cor:partition} (note that $\omega_{2}\circ\omega_{1}\circ\hat{\varphi}_{k}$ is a definable map). Moreover, since the fibres of $\mc{Z}$ are bounded, the same holds true for the fibres of $\mc{S}_{k}$. Hence, by Theorem \ref{thm:latticeest}, Lemma \ref{lem:LambdacapC}, and Equations (\ref{eq:estimate1}) and (\ref{eq:28}), we have
\begin{align}
\label{eq:estimate3} 
 & \left|\#(\Lambda\cap W)-\frac{\Vol\,W}{\det\Lambda}\right|\leq\nonumber \\
 & \sum_{k\in\mc{K}}\left|\#\left(\omega_{2}\circ\omega_{1}\circ\hat{\varphi}_{k}(\Lambda)\cap \omega_{2}\circ\omega_{1}\circ\hat{\varphi}_{k}\left(W\cap\hat{X}_{k}\right)\right)-\frac{\Vol\,\omega_{2}\circ\omega_{1}\circ\hat{\varphi}_{k}\left(W\cap\hat{X}_{k}\right)}{\det\omega_{2}\circ\omega_{1}\circ\hat{\varphi}_{k}(\Lambda)}\right|\nonumber \\
 & +\sum_{i=1}^{M}\#(\Lambda\cap W^{i})\ll_{\mc{W},\bs{\beta},\bs{\gamma}}\left(\sum_{k\in\mc{K}}\sum_{s=0}^{\mc{M}+\mc{N}-1}\frac{\left(\varepsilon^{\mc{B}}Q^{\mc{C}}\right)^{\frac{s}{\mc{B}+\mc{C}}}}{\lambda_{1}\left(\omega_{2}\circ\omega_{1}\circ\hat{\varphi}_{k}(\Lambda)\right)^{s}}\right)+\left(\frac{\diam(W\cap C)}{\lambda_{1}(\Lambda\cap C)}\right)^{c},
\end{align}
where $\lambda_{1}\left(\omega_{2}\circ\omega_{1}\circ\hat{\varphi}_{k}(\Lambda)\right)$ is the first successive minimum of the lattice $\omega_{2}\circ\omega_{1}\circ\hat{\varphi}_{k}(\Lambda)$.
\begin{prop}
\label{prop:firstmin}
Let $k\in\mc{K}$ and let $\lambda_{1}$ be the first successive minimum of the lattice \\
$\omega_{2}\circ\omega_{1}\circ\hat{\varphi}_{k}(\Lambda)$. Then,
$$\lambda_{1}\gg_{\bs{m},\bs{n},\bs{\beta},\bs{\gamma}}\min\left\{\nu(\Lambda,r),\ \left(\varepsilon^{\mc{B}} Q^{\mc{C}}\right)^{\frac{1}{\mc{B}+\mc{C}}}\frac{r}{\diam\,Z},\ \left(\varepsilon^{\mc{B}} Q^{\mc{C}}\right)^{\frac{1}{\mc{B}+\mc{C}}}\frac{\lambda_{1}(\Lambda\cap C)}{\diam(Z\cap C)}\right\}$$
for all $r>0$. By convention, the last term is $+\infty$ if $C=\{(\bs{0},\bs{0})\}$.
\end{prop}

We prove Proposition \ref{prop:firstmin} in Section \ref{sec:minimaproof}. Note that $C\subsetneq V$ implies $c\leq\mc{M}+\mc{N}-1$. Hence, combining (\ref{eq:estimate3}) and Proposition \ref{prop:firstmin}, we get that for all $r>0$
\begin{equation}
\left|\#(\Lambda\cap W)-\frac{\Vol\,W}{\det\Lambda}\right|\ll_{\mc{W},\bs{\beta},\bs{\gamma}}\#\mc{K}\left(1+\left(\frac{\left(\varepsilon^{\mc{B}} Q^{\mc{C}}\right)^{\frac{1}{\mc{B}+\mc{C}}}}{\nu(\Lambda,r)}+\frac{\diam\,Z}{r}+\frac{\diam(Z\cap C)}{\lambda_{1}(\Lambda\cap C)}\right)^{\mc{M}+\mc{N}-1}\right),\nonumber
\end{equation}
where the last term is null if $C=\{(\bs{0},\bs{0})\}$. It follows that
\begin{multline}
\left|\#(\Lambda\cap W)-\frac{\Vol\,W}{\det\Lambda}\right|\ll_{\mc{W},\bs{\beta},\bs{\gamma}}\nonumber \\
\inf_{0<r\leq \diam\,Z}\#\mc{K}\left(\frac{\left(\varepsilon^{\mc{B}} Q^{\mc{C}}\right)^{\frac{1}{\mc{B}+\mc{C}}}}{\nu(\Lambda,r)}+\frac{\diam\,Z}{r}+\frac{\diam(Z\cap C)}{\lambda_{1}(\Lambda\cap C)}\right)^{\mc{M}+\mc{N}-1}.\nonumber
\end{multline}
By Proposition \ref{prop:partition}, we have $\#\mc{K}\ll_{\bs{m},\bs{\beta}}\log\left(T/\varepsilon\right)^{M-1}$ and thus, the proof is complete.

\section{Proof of Proposition \ref{prop:partition}}
\label{sec:partitionproof}

In this section, we use the notation $S_{+}$ to indicate the set $S\setminus\bigcup_{i=1}^{M}\{\bs{x}_{i}=\bs{0}\}$ for any subset $S$ of $V_{\bs{m}}$.

\subsection{A partition of the boundary}

We start off by constructing a partition of the set $(\partial H)_{+}$ and a collection of linear maps defined on $V_{\bs{m}}$, that satisfy parts $i)-iv)$ for the set $(\partial H)_{+}$. Then, we extend this partition to $H_{+}$ by taking cones, and we prove that the same maps work for the whole set. 
To this end, we consider the sets
$$(\partial H)_{\textup{hyp}}:=\left\{\underline{\bs{x}}\in V_{\bs{m}}:\prod_{i=1}^{M}\left|\bs{x}_{i}\right|_{2}^{\beta_{i}}=\varepsilon^{\mc{B}},|\bs{x}_{1}|_{2},\dotsc,|\bs{x}_{M}|_{2}\leq T\right\}$$
and $(\partial H)_{\textup{non-hyp}}:=(\partial H)_{+}\setminus (\partial H)_{\textup{hyp}}$.

\subsection{The hyperbolic part}

We start by proving parts $i),iii)$, and $iv)$ for the set $(\partial H)_{\textup{hyp}}$. Let $\xi:\{\underline{\bs{x}}\in V_{\bs{m}}:\bs{x}_{i}\neq 0\ i=1,\dotsc,M\}\to\mb{R}^{M}$ with $\xi(\underline{\bs{x}})_{i}:=\log|\bs{x}_{i}|_{2}$ for $i=1,\dotsc,M$. We denote by $z_{i}$ the coordinates of the codomain of $\xi$. We also introduce the sets
$$\pi:=\left\{\bs{z}\in\mb{R}^{M}:\sum_{i=1}^{M}\beta_{i}z_{i}=\mc{B}\log\varepsilon\right\}$$
and
$$S:=\left(-\infty,\log T\right]^{M}.$$

\begin{lem}
\label{lem:toplev}
There exists a partition of $(\partial H)_{\textup{hyp}}$ of the form $(\partial H)_{\textup{hyp}}=\bigcup_{k\in\mc{K}_{\textup{hyp}}}\tilde{X}_{k}$, and there exists a collection of linear maps $\varphi_{k}:V_{\bs{m}}\to V_{\bs{m}}$ for $k\in\mc{K}_{\textup{hyp}}$, that satisfy parts $i),iii)$, and $iv)$ of Proposition \ref{prop:partition}.
\end{lem}

\begin{proof}
First, we observe that
$$(\partial H)_{\textup{hyp}}=\xi^{-1}(\pi\cap S)$$
Let $P$ be any point on the hyperplane $\pi$ and let $\{\bs{v}_{1},\dotsc,\bs{v}_{M-1}\}$ be an orthonormal basis of $\textup{lin}(\pi)$ (i.e., the only linear subspace associated to $\pi$). We consider a tiling of $\pi$ given by the sets
$$T_{\bs{k}}:=\{P+\lambda_{1}\bs{v}_{1}+\dotsb+\lambda_{M-1}\bs{v}_{M-1}:k_{i}\leq\lambda_{i}<k_{i}+1\ \mbox{for }i=1,\dotsc,M-1\}.$$
for $\bs{k}\in\mb{Z}^{M-1}$. Since $\pi\cap S$ is bounded and $\diam(T_{\bs{k}})\ll_{\bs{m}} 1$, we trivially have
\begin{equation}
\label{eq:triviallatticeest}
\#\{\bs{k}:T_{\bs{k}}\cap\pi\cap S\neq\emptyset\}\ll_{\bs{m}} \left(2+\diam(\pi\cap S)\right)^{\mc{M}-1}.
\end{equation}
Now, the set $\pi\cap S$ is a $(M-1)$-dimensional simplex, whose vertices $V^{i}$ ($i=1,\dotsc,M$) have coordinates
$$V^{i}_{h}:=
\begin{cases}\log T & \mbox{ if }h\neq i \\
\displaystyle{\frac{1}{\beta_{i}}\log\left(\frac{\varepsilon^{\mc{B}}}{T^{\mc{B}-\beta_{i}}}\right)} & \mbox{ if } h=i
\end{cases}$$
for $h=1,\dotsc,M$. We define $V^{0}:=\left(\log T,\dotsc,\log T\right)$, and we consider the only aligned box whose vertices include the points $V^{0}, V^{1},\dotsc,V^{M}$. We let $\bs{\mu}$ be its centre. Since the side length of this box is
$$\left|\frac{1}{\beta_{i}}\log\left(\frac{\varepsilon^{\mc{B}}}{T^{\mc{B}-\beta_{i}}}\right)-\log T\right|=\frac{\mc{B}}{\beta_{i}}\log\left(\frac{T}{\varepsilon}\right),$$
and since it contains $\pi\cap S$, by (\ref{eq:triviallatticeest}) we have
\begin{equation}
\label{eq:triviallatticeest1}
\#\{\bs{k}:T_{\bs{k}}\cap\pi\cap S\neq\emptyset\}\ll_{\bs{m},\bs{\beta}}\log\left(\frac{T}{\varepsilon}\right)^{M-1}.
\end{equation}
Now, we set
$$\mc{K}_{\textup{hyp}}:=\left\{\bs{k}\in\mb{Z}^{M}:T_{\bs{k}}\cap S\neq\emptyset\right\}.$$
Note that $T_{\bs{k}}\cap\pi\cap S=T_{\bs{k}}\cap S$.
Then, the sets $\{T_{\bs{k}}\cap S\}_{\bs{k}\in\mc{K}_{\textup{hyp}}}$ form a partition of $\pi\cap S$, and part $i)$ follows directly from (\ref{eq:triviallatticeest1}). We associate with each of these sets a translation $\tau_{\bs{k}}$ of the form $\tau_{\bs{k}}(\bs{z}):=\bs{z}+\bs{a}^{\bs{k}}$, where $\bs{a}^{\bs{k}}\in\mb{R}^{M}$. In particular, we choose $\bs{a}^{\bs{k}}$ to be the distance vector from the centre of the tile $T_{\bs{k}}$ to the point $C:=\left(\log\varepsilon,\dotsc,\log\varepsilon\right)\in\pi\cap S$. Given that $\bs{a}^{\bs{k}}\in\textup{lin}(\pi)$, we have
$$\sum_{i=1}^{M}\beta_{i}a^{\bs{k}}_{i}=0,$$
proving part $iiia)$.
Now, $C$ lies in the box of vertices $V_{0},\dotsc,V_{M}$, whereas the centre of the tile $T_{\bs{k}}$ lies in a box of centre $\bs{\mu}$ and side length at most $\max_{i}\frac{\mc{B}}{\beta_{i}}\log\left(\frac{T}{\varepsilon}\right)+\diam(T_{\bs{k}})$. Since $\diam(T_{\bs{k}})\ll_{\bs{m}}1$, we have
\begin{equation}
\left|\bs{a}^{\bs{k}}\right|\ll_{\bs{m},\bs{\beta}}\log\left(\frac{T}{\varepsilon}\right).\nonumber
\end{equation}
Hence, part $iiib)$ is proved (modulo the fact that $c$ depends uniquely on $\bs{m}$). Let $\tilde{X}_{\bs{k}}:=\xi^{-1}(T_{\bs{k}}\cap S)$ for $\bs{k}\in\mc{K}_{\textup{hyp}}$ and let $\varphi_{\bs{k}}:V_{\bs{m}}\to V_{\bs{m}}$ be the linear transformation defined by
$$\varphi_{\bs{k}}(\underline{\bs{x}})_{i}:=e^{a^{\bs{k}}_{i}}\bs{x}$$
for $i=1,\dotsc,M$. The following diagram commutes
\begin{equation}
\xymatrix{
\{\bs{x}_{i}\neq 0\ i=1,\dotsc,M\} \ar[d]^{\xi} \ar[r]^{\varphi_{\bs{k}}} & \{\bs{x}_{i}\neq 0\ i=1,\dotsc,M\} \ar[d]^{\xi} \\
\mb{R}^{M} \ar[r]^{\tau_{\bs{k}}} & \mb{R}^{M}
}.\nonumber
\end{equation}
Hence,
\begin{multline}
\varphi_{\bs{k}}(X_{\bs{k}})\subset\xi^{-1}\left\{|z_{i}|\leq\log\varepsilon+\diam(T_{\bs{k}})\ \mbox{for }i=1,\dotsc,M\right\} \\
\subset\left\{|\bs{x}_{i}|_{2}\leq\varepsilon e^{\diam(T_{\bs{k}})}\ \mbox{for }i=1,\dotsc,M\right\}.\nonumber
\end{multline}
To conclude the proof of part $iv)$ it suffices to rescale all coordinates by $e^{c}\ll_{\bs{m}}1$, where $c:=\diam(T_{\bs{k}})$.
\end{proof}

\begin{lem}
\label{lem:hypdefinable}
Each of the sets $\tilde{X}_{k}$ defined in Lemma \ref{lem:toplev} is definable.  
\end{lem}

\begin{proof}
Let $\textup{pr}_{\pi}:\mb{R}^{M}\to\pi$ be the orthogonal projection onto the hyperplane $\pi$, and let
$$\tilde{T}_{\bs{k}}:=\textup{pr}_{\pi}^{-1}\left(T_{\bs{k}}\right)$$
for $\bs{k}\in\mb{Z}^{M}$. Each of the sets $\tilde{T}_{\bs{k}}$ is semi-linear, since it is bounded by a finite number of hyperplanes. Moreover, we have
$$T_{\bs{k}}\cap S=\tilde{T}_{\bs{k}}\cap\pi\cap S.$$
It follows that all the inequalities defining the set $T_{\bs{k}}\cap S$ are semi-linear in $\bs{z}$. Let $\mf{L}(\bs{z})$ be the system defining $T_{\bs{k}}\cap S$. Then, $\tilde{X}_{\bs{k}}$ is defined by the system $\mf{L}(\xi(\underline{\bs{x}}))$, which is a system of inequalities of generalised polynomials\footnote{finite sums of monomials with non-negative real exponents. Note that the function $f(x)=x^r=\exp{(r\log x)}$ on $(0,+\infty)$ with real $r>0$ is definable in $\mb{R}_{\exp}$.} in the variables $\bs{x}_{i}$. Hence, the set $\tilde{X}_{k}$ is definable.
\end{proof}

\subsection{The non-hyperbolic part}

Now, we prove parts $i),iii)$, and $iv)$ for the set $(\partial H)_{\textup{non-hyp}}$.

\begin{lem}
\label{lem:lowerlevs}
There exists a partition of the set $(\partial H)_{\textup{non-hyp}}$ of the form $(\partial H)_{\textup{non-hyp}}= \\ \bigcup_{k\in\mc{K}_{\textup{non-hyp}}}\tilde{X}_{k}'$, and there exists a collection of linear maps $\varphi_{k}:V_{\bs{m}}\to V_{\bs{m}}$ for $k\in\mc{K}_{\textup{non-hyp}}$, that satisfy parts $i),iii)$, and $iv)$ of Proposition \ref{prop:partition}.
\end{lem}

\begin{proof}
Let $\bs{z}\in\xi((\partial H)_{\textup{non-hyp}})$. We define a unique point $\bs{z}^{*}\in\pi\cap S$ associated to $\bs{z}$ by the following procedure. By definition of $(\partial H)_{\textup{non-hyp}}$, we have
$$\sum_{i=1}^{M}\beta_{i}z_{i}<\mc{B}\log\varepsilon.$$
We increase the first coordinate $z_{1}$ of $\bs{z}$ until either $\sum_{i=1}^{M}\beta_{i}z_{i}=\mc{B}\log\varepsilon$ or $z_{1}=\log T$. We call the increased coordinate $z_{1}^{*}$. If
$$\beta_{1}z_{1}^{*}+\sum_{i=2}^{M}\beta_{i}z_{i}=\mc{B}\log\varepsilon,$$
we stop and we set $\bs{z}^{*}:=(z_{1}^{*},z_{2},\dotsc,z_{M})$. Otherwise, we increase the second coordinate $z_{2}$ until either $\beta_{1}z_{1}^{*}+\beta_{2}z_{2}+\sum_{i=3}^{M}\beta_{i}z_{i}=\mc{B}\log\varepsilon,$ or $z_{2}=\log T$. We call the increased coordinate $z_{2}^{*}$. If
$$\beta_{1}z_{1}^{*}+\beta_{2}z_{2}^{*}+\sum_{i=3}^{M}\beta_{i}z_{i}=\mc{B}\log\varepsilon,$$
we stop and we set $\bs{z}^{*}:=(z_{1}^{*},z_{2}^{*},z_{3},\dotsc,z_{M})$. Otherwise, we repeat the same steps for the remaining coordinates. This procedure terminates, since $\mc{B}\log T\geq\mc{B}\log\varepsilon$. Moreover, we have that $\bs{z}^{*}\in (\partial H)_{\textup{hyp}}$.
Now, we set $\mc{K}_{\textup{hyp}}=\mc{K}_{\textup{non-hyp}}$, and for each $k\in\mc{K}_{\textup{non-hyp}}$ we define
$$\tilde{X}_{k}':=\left\{\underline{\bs{x}}\in (\partial H)_{\textup{non-hyp}}:\xi(\underline{\bs{x}})^{*}\in\xi\left(\tilde{X}_{k}\right)=T_{k}\cap S\right\}.$$
Then, we have
$$(\partial H)_{\textup{non-hyp}}=\bigcup_{k\in\mc{K}_{\textup{hyp}}}\tilde{X}_{k}',$$
and this is a partition of $(\partial H)_{\textup{non-hyp}}$, since the sets $T_{k}\cap S$ form a partition of $\pi\cap S$. We show that the sets $\tilde{X}_{k}'$ and the maps $\varphi_{k}$ for $k\in\mc{K}_{\textup{non-hyp}}$ (i.e., the maps introduced in Lemma \ref{lem:toplev}) have the required properties. The proof of parts $i)$ and $iii)$ is trivial. To prove part $iv)$ we observe that, by construction, for each point $\underline{\bs{x}}\in\tilde{X}_{k}'$ there are points $\underline{\bs{y}}\in\tilde{X}_{k}$ such that $|\bs{x}_{i}|_{2}\leq |\bs{y}_{i}|_{2}$ for $i=1,\dotsc,M$ (e.g., any point $\underline{\bs{y}}\in\xi^{-1}\left(\xi(\underline{\bs{x}})^{*}\right)$). Therefore, since
$$\varphi_{k}\left(\tilde{X}_{k}\right)\subset\left\{|\bs{x}_{i}|_{2}\leq\varepsilon\ i=1,\dotsc,M\right\},$$
we have
$$\varphi_{k}\left(\tilde{X}_{k}'\right)\subset\left\{|\bs{x}_{i}|_{2}\leq\varepsilon\ i=1,\dotsc,M\right\},$$
by the definition of the maps $\varphi_{k}$.
\end{proof}

\begin{lem}
Each of the sets $\tilde{X}_{k}'$ defined in Lemma \ref{lem:lowerlevs} is definable.
\end{lem}

\begin{proof}
We have
\begin{multline}
\tilde{X}_{k}'=\left\{\underline{\bs{x}}\in (\partial H)_{\textup{non-hyp}}:\exists\underline{\bs{x}}^{*}\in\tilde{X}_{k}\ \mbox{such that }|\bs{x}_{i}^{*}|\geq|\bs{x}_{i}|\right. \\ 
\left.\mbox{and }|\bs{x}_{i}^{*}|>|\bs{x}_{i}|\Rightarrow(|\bs{x}_{h}^{*}|=T\ \mbox{for }h<i)\right\}.\nonumber
\end{multline}
Now, we consider the set
\begin{multline}
\tilde{X}_{k}'':=\left\{(\underline{\bs{x}},\underline{\bs{x}}^{*})\in (\partial H)_{\textup{non-hyp}}\times \tilde{X}_{k}:|\bs{x}_{i}^{*}|\geq|\bs{x}_{i}|\ \mbox{and }|\bs{x}_{i}^{*}|>|\bs{x}_{i}|\Rightarrow(|\bs{x}_{h}^{*}|=T\ \mbox{for }h<i)\right\},\nonumber
\end{multline}
and we let $\textup{pr}:V_{\bs{m}}\times V_{\bs{m}}\to V_{\bs{m}}$ be the projection onto the first cartesian factor. Then, $\tilde{X}_{k}'=\textup{pr}\left(\tilde{X}_{k}''\right).$ By the properties of $o$-minimal structures (see Definition \ref{def:definablestruc}), $\textup{pr}\left(\tilde{X}_{k}''\right)$ is a definable set.
\end{proof}

\subsection{From the boundary to the whole set}
Given a set $A\subset V_{\bs{m}}$, we denote by $\mc{C}(A)$ the cone generated by the set $A$, i.e., the set
$$\{t\underline{\bs{x}}:t\in(0,+\infty),\underline{\bs{x}}\in A\}.$$
Let $\mc{K}:=\mc{K}_{\textup{hyp}}\sqcup\mc{K}_{\textup{non-hyp}}$,
and let
$$X_{k}:=\mc{C}\left(\tilde{X}_{k}\right)\cap H_{+}$$
for $k\in\mc{K}$ (where we drop the apex $'$ for the sets $\tilde{X}_{k}'$ with $k\in\mc{K}_{\textup{non-hyp}}$). Then, clearly
$$H_{+}=\bigcup_{k\in\mc{K}}X_{k},$$
and this is a partition of the set $H_{+}$ (each line through the origin intersects the boundary at at most one point). We prove that the sets $X_{k}$ and the maps $\varphi_{k}$ satisfy parts $i)-iv)$ of Proposition \ref{prop:partition}. From Lemmas \ref{lem:toplev} and \ref{lem:lowerlevs}, we easily deduce
$$\#\mc{K}\ll_{\bs{m},\bs{\beta}}\log\left(\frac{T}{\varepsilon}\right)^{M-1}.$$ 
To prove part $ii)$, we need the following lemma.

\begin{lem}
\label{lem:cones}
Let $D\subset V_{\bs{m}}$ be a definable set. Then, the set $\mc{C}(D)\subset V_{\bs{m}}$ is also definable.
\end{lem}

\begin{proof}
We have
$$\mc{C}(D)=\left\{\underline{\bs{x}}\in V_{\bs{m}}: \exists t\in(0,+\infty)\ \mbox{such that }t\underline{\bs{x}}\in D\right\}.$$
We consider the set
$$\tilde{D}=\{(\underline{\bs{x}},t)\in V_{\bs{m}}\times\mb{R}:t\underline{\bs{x}}\in D,t>0\},$$
and we let $\textup{pr}:V_{\bs{m}}\times\mb{R}\to V_{\bs{m}}$ be the natural projection.
Then,
$$\mc{C}(D)=\textup{pr}\left(\tilde{D}\right).$$
By the properties of $o$-minimal structures (see Definition \ref{def:definablestruc}), $\textup{pr}\left(\tilde{D}\right)$ is a definable set.
\end{proof}

From Lemmas \ref{lem:hypdefinable} and \ref{lem:cones} it follows that $\mc{C}(\tilde{X}_{k})$ is a definable set for each $k$, proving part $ii)$. Part $iii)$ is a straightforward consequence of Lemmas \ref{lem:toplev} and \ref{lem:lowerlevs}. To prove part $iv)$, it suffices to note that for each point $\underline{\bs{x}}\in X_{k}$ there is a point $\underline{\bs{y}}\in\tilde{X}_{k}$ or $\underline{\bs{y}}\in\tilde{X}_{k}'$ such that $|\bs{x}_{i}|_{2}\leq|\bs{y}_{i}|_{2}$ for $i=1,\dotsc,M$ (namely $\{\underline{\bs{y}}\}=\{t\underline{\bs{x}}:t\in(0,+\infty)\}\cap(\partial H)_{+}$). Hence, part $iv)$ follows again from Lemmas \ref{lem:toplev} and \ref{lem:lowerlevs}, and by the definition of the maps $\varphi_{k}$.

\section{Proof of Proposition \ref{prop:firstmin}}
\label{sec:minimaproof}

Let $\underline{\bs{v}}\neq\bs{0}$ be a shortest vector in the lattice $\omega_{2}\circ\omega_{1}\circ\hat{\varphi}_{k}(\Lambda)$. Then, $\bs{v}$ has the form
\begin{equation}
\underline{\bs{v}}=\left(\theta\exp\left(a^{k}_{1}-c\right)\bs{x}_{1},\dotsc,\theta\exp\left(a^{k}_{M}-c\right)\bs{x}_{M},
\theta^{-\frac{\mc{B}}{\mc{C}}}\frac{Q}{Q_{1}}\bs{y}_{1},\dotsc,\theta^{-\frac{\mc{B}}{\mc{C}}}\frac{Q}{Q_{N}}\bs{y}_{N}\right)\nonumber
\end{equation}
for some point $(\underline{\bs{x}},\underline{\bs{y}})\in\Lambda$. It follows that
\begin{multline}
\label{eq:length}
|\underline{\bs{v}}|_{2}=\Bigg(\theta^{2}\exp\left(2a^{k}_{1}-2c\right)|\bs{x}_{1}|_{2}^{2}+\dotsb+\theta^{2}\exp\left(2a^{k}_{M}-2c\right)|\bs{x}_{M}|_{2}^{2}+ \\
\theta^{-\frac{2\mc{B}}{\mc{C}}}\frac{Q^{2}}{Q_{1}^{2}}|\bs{y}_{1}|_{2}^{2}+\dotsb+\theta^{-\frac{2\mc{B}}{\mc{C}}}\frac{Q^{2}}{Q_{N}^{2}}|\bs{y}_{N}|_{2}^{2}\Bigg)^{\frac{1}{2}}.
\end{multline}
Fix $r>0$. We consider three cases. Case $1$:
\begin{itemize}
\item $\bs{x}_{i}\neq\bs{0}$ for $i=1,\dotsc,M$ and $\bs{y}_{j}\neq\bs{0}$ for $j=1,\dotsc,N$;
\item $|(\underline{\bs{x}},\underline{\bs{y}})|_{2}< r$.
\end{itemize}
By applying the weighted arithmetic-geometric mean inequality to (\ref{eq:length}), with weights $\beta_{1},\dotsc,\beta_{M}$ and $\gamma_{1},\dotsc,\gamma_{N}$, we get
$$|\underline{\bs{v}}|_{2}\gg_{\bs{m},\bs{\beta},\bs{\gamma}}\left(\text{Nm}_{(\bs{\beta},\bs{\gamma})}(\underline{\bs{x}},\underline{\bs{y}})\right)^{\frac{1}{\mc{B}+\mc{C}}}\geq\nu(\Lambda,r),$$
where we used the fact that $\sum_{i=1}^{M}\beta_{i}a^{k}_{i}=0$ (see Proposition \ref{prop:partition}, part $iiib)$).

\noindent Case $2$:
\begin{itemize}
\item $|(\underline{\bs{x}},\underline{\bs{y}})|_{2}\geq r$.
\end{itemize}
In this case it must be either $|\bs{x}_{i_{0}}|_{2}\geq r/\sqrt{M+N}$ for some $1\leq i_{0}\leq M$ or $|\bs{y}_{j_{0}}|_{2}\geq r/\sqrt{M+N}$ for some $1\leq j_{0}\leq N$.

\noindent Case $2a$:
\begin{itemize}
\item there exists $1\leq i_{0}\leq M$ such that $|\bs{x}_{i_{0}}|_{2}\geq r/\sqrt{M+N}$.
\end{itemize}
By ignoring all the terms but $\bs{x}_{i_{0}}$, we get
$$|\underline{\bs{v}}|_{2}\gg_{\bs{m}}\theta e^{a^{\bs{k}}_{i_{0}}}|\bs{x}_{i_{0}}|_{2}\gg_{\bs{m},\bs{n}}\theta e^{a^{\bs{k}}_{i_{0}}}r.$$
It follows from Proposition \ref{prop:partition} part $iiia)$ that
$$|\underline{\bs{v}}|_{2}\gg_{\bs{m},\bs{n},\bs{\beta}}\frac{\left(\varepsilon^{\mc{B}} Q^{\mc{C}}\right)^{\frac{1}{\mc{B}+\mc{C}}}}{\varepsilon}\frac{\varepsilon}{T}r=\left(\varepsilon^{\mc{B}} Q^{\mc{C}}\right)^{\frac{1}{\mc{B}+\mc{C}}}\frac{r}{T}\geq\left(\varepsilon^{\mc{B}} Q^{\mc{C}}\right)^{\frac{1}{\mc{B}+\mc{C}}}\frac{r}{\diam\,Z}.$$
Case $2b$:
\begin{itemize}
\item there exists $1\leq j_{0}\leq N$ such that $|\bs{y}_{j_{0}}|_{2}\geq r/\sqrt{M+N}$.
\end{itemize}
By ignoring all the terms but $\bs{y}_{j_{0}}$, we get
$$|\underline{\bs{v}}|_{2}\gg_{\bs{m},\bs{n},\bs{\beta}}\theta^{-\frac{\mc{B}}{\mc{C}}}\frac{Q}{Q_{j_{0}}}r\geq\left(\varepsilon^{\mc{B}} Q^{\mc{C}}\right)^{\frac{1}{\mc{B}+\mc{C}}}\frac{r}{Q_{\max}}\geq\left(\varepsilon^{\mc{B}} Q^{\mc{C}}\right)^{\frac{1}{\mc{B}+\mc{C}}}\frac{r}{\diam\,Z},$$
where $Q_{\max}:=\max\{Q_{j}:j=1,\dotsc,N\}$.

\noindent Case $3$:
\begin{itemize}
\item $\bs{x}_{i_{0}}=\bs{0}$ for some $1\leq i_{0}\leq M$ or $\bs{y}_{j_{0}}=\bs{0}$ for some $1\leq j_{0}\leq N$.
\end{itemize}
We can suppose $C\neq\{(\bs{0},\bs{0})\}$, otherwise this case does not occur. Since $\Lambda$ is weakly admissible for $(S,C)$ we have that $(\underline{\bs{x}},\underline{\bs{y}})\in C$. Now, let
$$\delta_{\underline{\bs{x}}}:=
\begin{cases}
+\infty & \quad\text{ if }I=\{1,\dotsc,M\} \\
1 & \quad\text{ otherwise}
\end{cases},$$
$$\delta_{\underline{\bs{y}}}:=
\begin{cases}
+\infty & \quad\text{ if }J=\{1,\dotsc,N\} \\
1 & \quad\text{ otherwise}
\end{cases},$$
and let $Q_{C,\max}:=\max\left\{Q_{j}:j\notin J\right\}$ (if $J=\{1,\dotsc,N\}$, we set $Q_{C,\max}:=1$). Then, by Proposition \ref{prop:partition} part $iiia)$, we have
\begin{align}
 |\underline{\bs{v}}|_{2} & \geq\min\left\{\delta_{\underline{\bs{x}}}\theta\min_{i}{\exp\left(a^{\bs{k}}_{i}-c\right)},\ \delta_{\underline{\bs{y}}}\theta^{-\frac{\mc{B}}{\mc{C}}}\frac{Q}{Q_{C, \max}}\right\}|(\underline{\bs{x}},\underline{\bs{y}})|_{2}\nonumber \\
 & \gg_{\bs{m},\bs{\beta}}\min\left\{\delta_{\underline{\bs{x}}}\frac{\left(\varepsilon^{\mc{B}} Q^{\mc{C}}\right)^{\frac{1}{\mc{B}+\mc{C}}}}{\varepsilon}\frac{\varepsilon}{T},\ \delta_{\underline{\bs{y}}}\frac{\left(\varepsilon^{\mc{B}} Q^{\mc{C}}\right)^{\frac{1}{\mc{B}+\mc{C}}}}{Q_{C,\max}}\right\}\lambda_{1}(\Lambda\cap C)\nonumber \\
 & \geq\left(\varepsilon^{\mc{B}} Q^{\mc{C}}\right)^{\frac{1}{\mc{B}+\mc{C}}}\frac{\lambda_{1}(\Lambda\cap C)}{\diam(Z\cap C)}.\nonumber
\end{align}
This concludes the proof.

\section{Proof of Proposition \ref{cor:cor1}}

The set $Z$ that we consider in Proposition \ref{cor:cor1} has a slightly different structure from the fibres of the family $\mc{Z}$ appearing in Theorem \ref{thm:mainestimate}. In particular, it involves the maximum norm $|\cdot|_{\infty}$ instead of the Euclidean norm $|\cdot|_{2}$. Therefore, in order to apply Theorem \ref{thm:mainestimate} to the set $Z$, we need to introduce a new family $\mc{W}$ and see $Z$ as a fibre of $\mc{W}$. Let $\bs{m}=\bs{\beta}:=(1,\dotsc,1)\in\mb{R}^{\mc{M}}$ and let $\bs{n}=\bs{\gamma}:=\mc{N}$ (which implies $M=\mc{M}$ and $N=1$ according to the notation described in the Introduction). We set $\mc{W}:=\mc{H}\times\mc{R}^{\infty}$, where
$$\mc{H}:=\left\{(\underline{\bs{x}},\varepsilon',T')\in V_{\bs{m}}\times(0,+\infty)^{2}:\textup{Nm}_{\bs{m}}(\underline{\bs{x}})^{\frac{1}{\mc{M}}}<\varepsilon',\ |x_{i}|\leq T'\ i=1,\dotsc,M\right\},$$
and 
$$\mc{R}^{\infty}:=\left\{(\bs{y},Q')\in V_{\bs{n}}\times\mb{R}:|\bs{y}|_{\infty}\leq Q'\right\}$$
(note that the definition of $\mc{H}$ hasn't changed). Then, $Z=W_{\bs{\tau}}$, where
$$\bs{\tau}:=\left(\varepsilon',T',Q'\right)=\left(\varepsilon^{\frac{1}{\mc{M}}},T,Q\right).$$
To prove proposition \ref{cor:cor1}, we need to estimate
\begin{equation}
\label{eq:intersection}
\#\left(M(\bs{L},\varepsilon,T,Q)\right)=\#\left((\Lambda_{\bs{L}}\cap W_{\bs{\tau}})\setminus C\right),
\end{equation}
where $C:=\left\{\bs{y}=\bs{0}\right\}\subset V$. We consider two different cases. First, we assume
\begin{equation}
\label{eq:case1}
\varepsilon Q^{\mc{N}}/\phi(Q)\geq 1.
\end{equation}
In this case, we use Theorem \ref{thm:mainestimate} to estimate $\#(\Lambda_{\bs{L}}\cap W_{\bs{\tau}})$. A suitable choice for the parameter $\bs{\eta}(\bs{\tau})$ in order to have $W_{\bs{\tau}}\subset Z_{\bs{\eta}(\bs{\tau})}$ is $\bs{\eta}(\bs{\tau})=\left(\varepsilon^{\frac{1}{\mc{M}}},T,\sqrt{\mc{N}}Q\right)$. Also, we need to show that the lattice $\Lambda_{\bs{L}}$ is weakly admissible for the couple $(\mc{S},C)$, where $\mc{S}:=((\bs{m},\bs{n}),(\bs{\beta},\bs{\gamma}))$. We do this in the following lemma.
\begin{lem}
\label{lem:weakadmcor}
Let $\bs{m}=\bs{\beta}:=(1,\dotsc,1)\in\mb{R}^{\mc{M}}$ and let $\bs{n}=\bs{\gamma}:=\mc{N}$. Let also $\mc{S}:=((\bs{m},\bs{n}),(\bs{\beta},\bs{\gamma}))$ and let $C:=\{\bs{y}=\bs{0}\}$. Then,
\begin{equation}
\label{eq:weakadmcor}
\nu(\Lambda_{\bs{L}},\varrho)\geq\phi(\varrho)^{\frac{1}{\mc{M}+\mc{N}}}
\end{equation}
for all $\varrho>0$. Therefore, the lattice $\Lambda_{\bs{L}}$ is weakly admissible for the couple $(\mc{S},C)$ (see Definition \ref{def:weakadmin}).
\end{lem}

\begin{proof}
Let $\varrho\in(0,+\infty)$. If $\varrho\leq\lambda_{1}(\Lambda_{\bs{L}}\setminus C)$, then $\nu(\Lambda_{\bs{L}},\varrho)=+\infty$ and (\ref{eq:weakadmcor}) holds true. We can thus suppose that $\varrho>\lambda_{1}(\Lambda_{\bs{L}}\setminus C)$. Let $\bs{v}\in\Lambda_{\bs{L}}\setminus C$ with $|\bs{v}|_{2}<\varrho$. Then,
$$\bs{v}=(L_{1}(\bs{q})+p_{1},\dotsc,L_{\mc{M}}(\bs{q})+p_{\mc{M}},\bs{q})$$
for some $\bs{p}\in\mb{Z}^{\mc{M}}$ and $\bs{q}\in\mb{Z}^{\mc{N}}\setminus\{\bs{0}\}$. It follows from the hypothesis that
$$\textup{Nm}_{\bs{\beta},\bs{\gamma}}(\bs{v})=|\bs{q}|_{2}^{\mc{N}}\prod_{i=1}^{\mc{M}}\left|L_{i}(\bs{q})+p_{i}\right|\geq|\bs{q}|_{\infty}^{\mc{N}}\prod_{i=1}^{\mc{M}}\left\|L_{i}(\bs{q})\right\|\geq\phi(|\bs{q}|_{\infty})\geq\phi(\varrho),$$
where we used the fact that $\phi$ is non-increasing. Hence, $\nu(\Lambda_{\bs{L}},\varrho)\geq\phi(\varrho)^{1/(\mc{M}+\mc{N})}$.
\end{proof}

By applying Theorem \ref{thm:mainestimate} to $W_{\bs{\tau}}\subset Z_{\bs{\eta}(\bs{\tau})}$, we find

\begin{multline}
\label{eq:cor1eq1}
\left|\#(\Lambda_{\bs{L}}\cap Z)-\Vol\,Z\right|\ll_{\mc{M},\mc{N}} \\
\inf_{0<r\leq\diam\left(Z_{\bs{\eta}(\bs{\tau})}\right)}\log\left(\frac{T^{\mc{M}}}{\varepsilon}\right)^{\mc{M}-1}\left(\left(\frac{\varepsilon Q^{\mc{N}}}{\phi(r)}\right)^{\frac{1}{\mc{M}+\mc{N}}}+\frac{T+Q}{r}+\frac{T}{\lambda_{1}(\Lambda_{\bs{L}}\cap C)}\right)^{\mc{M}+\mc{N}-1}.
\end{multline}
Now, since $\Lambda_{\bs{L}}\cap C=\mb{Z}^{\mc{M}}\times\{\bs{0}\}$, we have $\lambda_{1}(\Lambda_{\bs{L}}\cap C)=1$. Hence, by choosing $r=Q$ in (\ref{eq:cor1eq1}), we deduce 
\begin{equation}
\label{eq:cor1eq2}
\left|\#(\Lambda_{\bs{L}}\cap Z)-\Vol\,Z\right|\ll_{\mc{M},\mc{N}}\log\left(\frac{T^{\mc{M}}}{\varepsilon}\right)^{\mc{M}-1}\left(\left(\frac{\varepsilon Q^{\mc{N}}}{\phi(Q)}\right)^{\frac{1}{\mc{M}+\mc{N}}}+1+T\right)^{\mc{M}+\mc{N}-1}.
\end{equation}
An easy integration shows that
$$\Vol\,Z=2^{\mc{M}+\mc{N}} Q^{\mc{N}}\left[\varepsilon\log\left(\frac{T^{\mc{M}}}{\varepsilon}\right)^{\mc{M}-1}+T^{\mc{M}}\left(1-\left(1-\frac{\varepsilon}{T^{\mc{M}}}\right)^{\mc{M}-1}\right)\right].$$
Thus, (\ref{eq:intersection}) and (\ref{eq:cor1eq2}) imply
\begin{multline}
\label{eq:cor1eq3}
\left|\#M(\bs{L},\varepsilon,T,Q)-2^{\mc{M}+\mc{N}} Q^{\mc{N}}\left[\varepsilon\log\left(\frac{T^{\mc{M}}}{\varepsilon}\right)^{\mc{M}-1}+T^{\mc{M}}\left(1-\left(1-\frac{\varepsilon}{T^{\mc{M}}}\right)^{\mc{M}-1}\right)\right]\right|\\
\ll_{\mc{M},\mc{N}}\#(\Lambda_{\bs{L}}\cap Z\cap C)+\log\left(\frac{T^{\mc{M}}}{\varepsilon}\right)^{\mc{M}-1}\left(\left(\frac{\varepsilon Q^{\mc{N}}}{\phi(Q)}\right)^{\frac{1}{\mc{M}+\mc{N}}}+1+T\right)^{\mc{M}+\mc{N}-1}.
\end{multline}
Since $\#(\Lambda_{\bs{L}}\cap Z\cap C)\leq (2T+1)^{\mc{M}}$, the required estimate is a straightforward consequence of (\ref{eq:case1}) and (\ref{eq:cor1eq3}).

We are now left to prove the claim for $\varepsilon Q^{\mc{N}}/\phi(Q)<1$.

\begin{lem}
\label{lem:emptycase}
Suppose that $\varepsilon Q^{\mc{N}}/\phi(Q)<1$. Then, $\Lambda_{\bs{L}}\cap Z\subset C$.
\end{lem}

\begin{proof}
Assume by contradiction that there exists $\underline{\bs{v}}\in(\Lambda_{\bs{L}}\cap Z)\setminus C$. Then,
$$\underline{\bs{v}}=(L_{1}(\bs{q})+p_{1},\dotsc,L_{\mc{M}}(\bs{q})+p_{\mc{M}},\bs{q})$$
for some $\bs{p}\in\mb{Z}^{\mc{M}}$ and $\bs{q}\in\mb{Z}^{\mc{N}}\setminus\{\bs{0}\}$. Since $\underline{\bs{v}}\in Z$, we have
$$|\bs{q}|_{\infty}^{\mc{N}}\prod_{i=1}^{\mc{M}}\left\|L_{i}(\bs{q})\right\|\leq|\bs{q}|_{\infty}^{\mc{N}}\prod_{i=1}^{\mc{M}}\left|L_{i}(\bs{q})+p_{i}\right|\leq Q^{\mc{N}}\varepsilon<\phi(Q),$$
and this contradicts (\ref{eq:semimultbadapp0}).
\end{proof}

If $\varepsilon Q^{\mc{N}}/\phi(Q)< 1$, it follows from Lemma \ref{lem:emptycase} and (\ref{eq:intersection}) that $M(\bs{L},\varepsilon,T,Q)=\emptyset$. Hence, to prove Proposition \ref{cor:cor1}, it suffices to show that
\begin{equation}
\label{eq:31}
2^{\mc{M}+\mc{N}}\log\left(\frac{T^{\mc{M}}}{\varepsilon}\right)^{\mc{M}-1}\varepsilon Q^{\mc{N}}\ll_{\mc{M},\mc{N}}(1+T)^{\mc{M}+\mc{N}-1}\log\left(\frac{T^{\mc{M}}}{\varepsilon}\right)^{\mc{M}-1}\left(\frac{\varepsilon Q^{\mc{N}}}{\phi(Q)}\right)^{\frac{\mc{M}+\mc{N}-1}{\mc{M}+\mc{N}}},
\end{equation}
and that
\begin{equation}
\label{eq:32}
2^{\mc{M}+\mc{N}} Q^{\mc{N}}T^{\mc{M}}\left(1-\left(1-\frac{\varepsilon}{T^{\mc{M}}}\right)^{\mc{M}-1}\right)\ll_{\mc{M},\mc{N}}(1+T)^{\mc{M}+\mc{N}-1}\log\left(\frac{T^{\mc{M}}}{\varepsilon}\right)^{\mc{M}-1}\left(\frac{\varepsilon Q^{\mc{N}}}{\phi(Q)}\right)^{\frac{\mc{M}+\mc{N}-1}{\mc{M}+\mc{N}}}.
\end{equation}

Inequality (\ref{eq:31}) follows immediately from the assumption $\varepsilon Q^{\mc{N}}/\phi(Q)<1$. To prove (\ref{eq:32}), we notice that
\begin{equation}
\label{eq:33}
1-\left(1-\frac{\varepsilon}{T^{\mc{M}}}\right)^{\mc{M}-1}\ll_{\mc{M}}\frac{\varepsilon}{T^{\mc{M}}},
\end{equation}
and again we use the fact that $\varepsilon Q^{\mc{N}}/\phi(Q)<1$. The proof is hence complete.

\section{Proof of Corollary \ref{cor:cor2}}

We notice that
\begin{align}
\label{eq:cor2eq1} 
 & \sum_{\substack{\bs{q}\in[-Q,Q]^{\mc{N}}\\ \cap\ \mb{Z}^{\mc{N}}\setminus\{\bs{0}\}}}\prod_{i=1}^{\mc{M}}\|L_{i}(\bs{q})\|^{-1}\nonumber \\
 & \leq\sum_{k=0}^{\infty}2^{k+1}\#\left\{\bs{q}\in[-Q,Q]^{\mc{N}}\cap\mb{Z}^{\mc{N}}\setminus\{\bs{0}\}:2^{-k-1}\leq\prod_{i=1}^{\mc{M}}\|L_{i}(\bs{q})\|<2^{-k}\right\}\nonumber \\
 & \leq\sum_{k=0}^{\infty}2^{k+1}\#\left\{\bs{q}\in[-Q,Q]^{\mc{N}}\cap\mb{Z}^{\mc{N}}\setminus\{\bs{0}\}:\prod_{i=1}^{\mc{M}}\|L_{i}(\bs{q})\|<2^{-k}\right\}\nonumber \\
 & =\sum_{k=0}^{\infty}2^{k+1}\# M\left(\bs{L},2^{-k},\frac{1}{2},Q\right)\nonumber \\
 & =\sum_{k=0}^{\left\lfloor\log_{2}\left(\frac{Q^{\mc{N}}}{\phi(Q)}\right)\right\rfloor}2^{k+1}\# M\left(\bs{L},2^{-k},\frac{1}{2},Q\right),
\end{align}
where the last equality follows from (\ref{eq:intersection}) and Lemma \ref{lem:emptycase} ($\varepsilon=2^{-k}$).

We use Proposition \ref{cor:cor1} to estimate the right-hand side of (\ref{eq:cor2eq1}). Since we need $T^{\mc{M}}/\varepsilon\geq e^{\mc{M}}$, i.e., $2^{k-\mc{M}}\geq e^{\mc{M}}$, we split the sum into two parts, one for $2^{k-\mc{M}}<e^{\mc{M}}$ and one for $2^{k-\mc{M}}\geq e^{\mc{M}}$. We find
\begin{align}
 & \sum_{\substack{\bs{q}\in[-Q,Q]^{\mc{N}}\\ \cap\ \mb{Z}^{\mc{N}}\setminus\{\bs{0}\}}}\prod_{i=1}^{\mc{M}}\|L_{i}(\bs{q})\|^{-1}\leq\sum_{k=0}^{\left\lfloor\mc{M}\left(1+1/\log 2\right)\right\rfloor}2^{k+1}\# M\left(\bs{L},2^{-k},\frac{1}{2},Q\right)\nonumber \\
 & +\sum_{k=\left\lceil\mc{M}\left(1+1/\log 2\right)\right\rceil}^{\left\lfloor\log_{2}\left(\frac{Q^{\mc{N}}}{\phi(Q)}\right)\right\rfloor}2^{k+1}\# M\left(\bs{L},2^{-k},\frac{1}{2},Q\right)\nonumber \\
 & \ll_{\mc{M},\mc{N}}Q^{\mc{N}}+\sum_{k=\left\lceil\mc{M}\left(1+1/\log 2\right)\right\rceil}^{\left\lfloor\log_{2}\left(\frac{Q^{\mc{N}}}{\phi(Q)}\right)\right\rfloor}2^{k+1}(k-\mc{M})^{\mc{M}-1}\left(2^{-k}Q^{\mc{N}}+\left(\frac{2^{-k}Q^{\mc{N}}}{\phi(Q)}\right)^{\frac{\mc{M}+\mc{N}-1}{\mc{M}+\mc{N}}}\right)\label{eq:cor2eq2.1}
\end{align}
\begin{align}
& \ll_{\mc{M},\mc{N}}\sum_{k=0}^{\left\lfloor\log_{2}\left(\frac{Q^{\mc{N}}}{\phi(Q)}\right)\right\rfloor}k^{\mc{M}-1}\left(Q^{\mc{N}}+2^{\frac{k}{\mc{M}+\mc{N}}}\left(\frac{Q^{\mc{N}}}{\phi(Q)}\right)^{\frac{\mc{M}+\mc{N}-1}{\mc{M}+\mc{N}}}\right)\label{eq:cor2eq2.2},
\end{align}
where we estimate $\# M\left(\bs{L},2^{-k},1/2,Q\right)$ with $Q^{\mc{N}}$ for $k\leq \left\lfloor\mc{M}\left(1+1/\log 2\right)\right\rfloor$. Note that we used (\ref{eq:33}) to obtain (\ref{eq:cor2eq2.1}). The required result follows from (\ref{eq:cor2eq2.2}) combined with the trivial estimates $\sum_{k=0}^{K}k^{\mc{M}-1}\leq K^{\mc{M}}$ and $\sum_{k=0}^{K}k^{\mc{M}-1}2^{\frac{k}{\mc{M}+\mc{N}}}\ll_{\mc{M},\mc{N}}K^{\mc{M}-1}2^{\frac{K}{\mc{M}+\mc{N}}}$.

\section{Proof of Propositions \ref{thm:thm3refined} and \ref{thm:thm2refined}}
\label{sec:proofHL}

The proofs of Theorems \cite[Theorems 2 and 3]{HuangLiu:Simultaneous} rely on \cite[Theorems 7 and 8]{HuangLiu:Simultaneous}. We state here the refined versions of Theorems 7 and 8. Let $\tilde{\bs{a}}=(q,\bs{a})$, where $q\in\mb{Z}$ and $\bs{a}\in\mb{Z}^{d}$ . For $0<\delta<1/2$ we set
$$\mc{A}(q,\delta):=\#\left\{\bs{a}\in\mb{Z}^{d}:\bs{a}\in(0,q]^{d},\ \left\|\tilde{\bs{a}}\tilde{\bs{A}}\right\|<\delta\right\},$$
and
$$\mc{N}(Q,\delta):=\sum_{q=1}^{Q}\mc{A}(q,\delta),$$
where $Q\in\mb{N}$.

\begin{lem}
\label{thm:HLthm8}
Let $\bs{A}\in\mb{R}^{d\times(\mc{N}-d)}$ and let $\bs{\alpha}_{0}\in\mb{R}^{\mc{N}-d}$. Let $\tilde{\bs{A}}$ be the matrix defined in (\ref{eq:tilde}). Assume that $\tilde{\bs{A}}$ is $\tilde{\phi}$-semi multiplicatively badly approximable, where $\tilde{\phi}:[1,+\infty)\to(0,1]$ is such that $\tilde{\phi}(\lambda x)\gg_{\lambda}\tilde{\phi}(x)$ for all $\lambda\gg 1$. Then, for all $\varepsilon'>0$ we have
\begin{multline}\left|\mc{N}(Q,\delta)-2^{\mc{N}-d}\delta^{\mc{N}-d}\sum_{q=1}^{Q}q^{d}\right|\leq\varepsilon'\delta^{\mc{N}-d}Q^{d+1}+ \\
O_{\varepsilon',\mc{N},d}\left(\log\left(\frac{1/\delta^{\mc{N}-d}}{\tilde{\phi}(1/\delta)}\right)^{d+1}+\frac{1}{\tilde{\phi}(1/\delta)}\log\left(\frac{1/\delta^{\mc{N}-d}}{\tilde{\phi}(1/\delta)}\right)^{d}\right). \nonumber
\end{multline}
\end{lem}

\begin{lem}
\label{thm:HLthm7}
Let $\bs{A}\in\mb{R}^{d\times(\mc{N}-d)}$ be a $\phi$-semi multiplicatively badly approximable matrix, where $\phi:[1,+\infty)\to(0,1]$ is such that $\phi(\lambda x)\gg_{\lambda}\phi(x)$ for all $\lambda\gg 1$. Then, for all $\varepsilon'>0$ we have
$$\left|\mc{A}(q,\delta)-2^{\mc{N}-d}\delta^{\mc{N}-d}q^{d}\right|\leq\varepsilon'\delta^{\mc{N}-d}q^{d}+O_{\varepsilon',\mc{N},d}\left(\log\left(\frac{1/\delta^{\mc{N}-d}}{\phi(1/\delta)}\right)^{d}+\frac{1}{\phi(1/\delta)}\log\left(\frac{1/\delta^{\mc{N}-d}}{\phi(1/\delta)}\right)^{d-1}\right).$$
\end{lem}

For simplicity, we prove Lemma \ref{thm:HLthm7} first.

\begin{proof}
From Huang and Liu's proof of \cite[Theorem 7]{HuangLiu:Simultaneous}, we have
$$\mc{A}(q,\delta)\leq\left(2\delta+\frac{1}{J+1}\right)^{\mc{N}-d}\left(q^{d}+\sum_{0<|\bs{j}|_{\infty}\leq J}\prod_{u=1}^{d}\left\|A_{u}(\bs{j})\right\|^{-1}\right)$$
for any $J\in\mb{N}$ (recall that $A_{u}$ denotes the linear form induced by the $u$-th row of the matrix $\bs{A}$). We apply Corollary \ref{cor:cor2} to estimate the right-hand side. We conclude the proof as in \cite{HuangLiu:Simultaneous}, by using the fact that $\phi(\kappa/\delta)\gg_{\kappa}\phi(1/\delta)$, where $\kappa$ is some large integer.
\end{proof}

\noindent The proof of Lemma \ref{thm:HLthm8} is along the same lines.

Now, we show how to prove Proposition \ref{thm:thm2refined}. We follow \cite{HuangLiu:Simultaneous}. First, we note that without loss of generality we can assume $\psi(x)\geq\hat{\psi}(x)$ for all $x$, since otherwise we replace $\psi$ with $\max\left\{\hat{\psi}(x),\psi(x)\right\}$, and we prove that the Hausdorff dimension of the set \\ $\mathscr{S}_{\mc{N}}\left(\max\left\{\hat{\psi}(x),\psi(x)\right\}\right)\supset\mathscr{S}_{\mc{N}}(\psi)$ is zero (here we use condition $iiia)$). It follows that in condition $iiib)$ we can replace $\hat{\psi}$ with $\psi$.
To prove the claim, we need to estimate $\mc{A}(q,C\psi(q))$, where $C$ is some large constant depending on $\bs{A}$ (see \cite[Proof of Thm. 2]{HuangLiu:Simultaneous}). By applying Lemma \ref{thm:HLthm7} with $\varepsilon'=1$, we find
\begin{multline}
\label{eq:final1}
\left|\mc{A}(q,C\psi(q))-(2C)^{\mc{N}-d}\psi(q)^{\mc{N}-d}q^{d}\right| \\
\ll_{C,\mc{N},d}\psi(q)^{\mc{N}-d}q^{d}+
\log\left(\frac{1/\psi(q)^{\mc{N}-d}}{\phi(1/\psi(q))}\right)^{d}+\frac{1}{\phi(1/\psi(q))}\log\left(\frac{1/\psi(q)^{\mc{N}-d}}{\phi(1/\psi(q))}\right)^{d-1}.
\end{multline}
Then, from $ii)$ we deduce
\begin{equation}
\label{eq:final2}
\log\left(\frac{1/\psi(q)^{\mc{N}-d}}{\phi(1/\psi(q))}\right)^{d}\ll_{\mc{N},d,\gamma}\log\left(\frac{1}{\psi(q)}\right)^{d}\ll_{\mc{N},d}\frac{1}{\phi(1/\psi(q))}\log\left(\frac{1/\psi(q)^{\mc{N}-d}}{\phi(1/\psi(q))}\right)^{d-1}.
\end{equation}
Finally, condition $iiib)$ with $\psi$ in lieu of $\hat{\psi}$ implies
\begin{equation}
\label{eq:final3}
\frac{1}{\phi(1/\psi(q))}\log\left(\frac{1/\psi(q)^{\mc{N}-d}}{\phi(1/\psi(q))}\right)^{d-1}\ll_{\mc{N},d,s}\frac{\psi(q)^{\mc{N}-d}q^{d}}{\log(q)^{d-1}}\log\left(\frac{q^{d}}{\log(q)^{d-1}}\right)^{d-1}\ll_{d}\psi(q)^{\mc{N}-d}q^{d}. 
\end{equation}
Hence, from (\ref{eq:final1}), (\ref{eq:final2}), and (\ref{eq:final3}) we deduce $\mc{A}(q,C\psi(q))\ll_{C,\mc{N},d,s,\gamma}\psi(q)^{\mc{N}-d}q^{d}$, and we can conclude just as in \cite{HuangLiu:Simultaneous}.

To prove Proposition \ref{thm:thm3refined}, we use Lemma \ref{thm:HLthm8} and part $iii)$ to obtain an estimate of \\ $\mc{N}(Q,C\psi(Q))$.

\section*{Acknowledgements}
My deep gratitude goes to my supervisor, Martin Widmer, for his valuable advice and constant encouragement. I would also like to thank Royal   
Holloway, University of London, for funding my position here.

\addcontentsline{toc}{section}{\bibname}
\bibliographystyle{plain}
\bibliography{Bibliography}

\end{document}